\documentclass[12pt]{amsart}
\usepackage{amsfonts}
\usepackage{amsmath}
\usepackage{amsxtra}
\usepackage{amssymb,latexsym}
\usepackage[mathcal]{eucal}
\usepackage{amscd}
\usepackage{hyperref}

\newtheorem{theo}{{\bfseries Theorem}}[section]
\newtheorem{prop}[theo]{{\bfseries Proposition}}
\newtheorem{lem}[theo]{{\bfseries Lemma}}
\newtheorem{cor}[theo]{{\bfseries Corollary}}
\newtheorem{df}[theo]{{\bfseries Definition}}

\newtheorem{exes}[theo]{{\bfseries Examples}}

\def \ol {\overline}
\def \N {\mathbb N}

\def \Z {\mathbb Z}
\def \R {\mathbb R}
\def \Q {\mathbb Q}

\def \B {\mathcal B}
\def \CC {\mathcal C}

\def \L {\mathcal L}

\def \J {\mathcal J}

\def \M {\mathcal M}

\def \P {\mathcal P}

\def \S {\mathcal S}

\def \lm {\lambda}
\def \ep {\epsilon}

\def \th {\theta}
\def \d {\delta}

\def \s {\sigma}

\def \1  {{\mathbf 1}}
\def \0  {{\mathbf 0}}

\usepackage{amssymb,latexsym}
\usepackage[mathcal]{eucal}
\usepackage{amscd}
\usepackage{amsmath}
\usepackage{graphicx}
\usepackage{amsfonts}
\usepackage{amscd}

\numberwithin{equation}{section}
\begin{document}

\begin{titlepage}
\large
\title {\bfseries Invertibility, Often}
\author{Ethan Akin}
 \vspace{.7cm}

\address{Mathematics Department \\
    The City College \\ 137 Street and Convent Avenue \\
       New York City, NY 10031, USA     }
\email{ethanakin@earthlink.net}

\author{Benjamin Weiss}
 \vspace{.7cm}

\address{Institute of Mathematics \\ Hebrew University of Jerusalem \\ Jerusalem, Israel    }
\email{weiss@math.huji.ac.il}

\date{July, 2026}

{\footnotesize \begin{abstract} By using a similar pattern of arguments, we show that in four categories the collection of isomorphisms forms a residual
subset of the space of morphisms. We first consider surjective continuous mappings on Cantor spaces. Next, we look at measure preserving maps on
Polish measure spaces. We then consider the $L^1$ representations of nonsingular maps on
Polish measure spaces. Finally, we examine continuous, measure 
preserving maps on Cantor spaces equipped with so-called good measures.

\end{abstract}}

\keywords{ Cantor space, residual subset, invertible mappings, measure spaces, measure algebras, nonsingular mappings, isometries,
good measures, clopen values set }

\thanks{{\em 2010 Mathematical Subject Classification} 54C05, 28C15, 28D05}

\end{titlepage}
\maketitle

\tableofcontents

\vspace{1cm}
\setcounter{page}{1}
\section{ \textbf{Introduction}}\vspace{.25cm}

Consider the \emph{tent map} $T : [0,1] \to [0,1]$, defined by
\begin{equation}\label{eq1.01}
T(t) = \begin{cases} 2t \ \ \ \text{for} \ 0 \le t \le \frac{1}{2}, \\ 2(1 - t) \ \text{for} \  \frac{1}{2} \le t \le 1. \end{cases}
\end{equation}

This is a surjective, continuous map. It is clear that it cannot be uniformly approximated by a homeomorphism. If $J$ is a non-trivial closed
interval contained in $(0,1)$, then $T^{-1}(J)$ is union of two disjoint intervals $J_1 \subset (0,\frac{1}{2})$
and $J_2 \subset (\frac{1}{2},1)$, with each mapping onto $J$.
If $T_1 : [0,1] \to [0,1]$ is a continuous map close to $T$, then $T_1(J_1)$ and $T_1(J_2)$ are each intervals close to $J$. They may not
be equal but they will still overlap and so $T_1$ will not be injective.

One can perform this folding procedure for the Cantor set as well. In addition, the tent map is measure preserving with respect to the Lebesgue measure
$\lm$ on $[0,1]$, i.e. $\lm(T^{-1}(E)) = \lm(E)$ for every Borel subset of $[0,1]$. As we will see, in each of these categories, $T$ can be approximated arbitrarily
closely by a suitable invertible map.

We will consider four different categories. \vspace{.25cm}
\begin{itemize}
\item[Section 2:] The objects are Cantor spaces and the morphisms are surjective continuous maps, with the topology of uniform convergence on the
space of maps.

\item[Section 3:] The objects are finite or $\s$-finite measure spaces obtained from complete,
separable metric spaces and the morphisms are measure preserving measurable maps.  The topology is that of pointwise convergence on the Borel sets.

\item[Section 4:] The objects are finite measure spaces obtained from complete,
separable metric spaces and the morphisms are nonsingular measurable maps, using the representation by isometries on the associated $L^1$ spaces.

\item[Section 5:] The objects are again Cantor spaces but equipped with certain special ''good'' measures. The morphisms are continuous, measure preserving maps
and the topology is that of uniform convergence.
\end{itemize}

In each case, we show that the invertible maps, i.e. the isomorphisms in the category, comprise a dense $G_{\d}$ subset of the space of maps.  In each case,
the space of maps is a complete metric space and so the Baire Category Theorem applies.

The argument in each case consists of three steps. We illustrate them with the Cantor set case.\vspace{.25cm}

\begin{itemize}
\item[Uniqueness:]  Any two Cantor spaces are homeomorphic and so, in particular, any two nonempty clopen (=simultaneously closed and open)
subsets of Cantor spaces are homeomorphic.

\item[$G_{\d}$:] In the space of surjective continuous maps between Cantor spaces, the invertible maps comprise a $G_{\d}$ subset.

\item[Density:]  In the space of surjective continuous maps between Cantor spaces, the invertible maps form a dense subset.
\end{itemize}

In the other categories, these same three steps of Uniqueness, $G_{\d}$, and density occur in the appropriate, albeit more complicated, ways.\vspace{.5cm}

\textbf{Notes:} The importance of those dense sets which are $G_{\d}$ comes from the \emph{Baire Category Theorem} which says that in a completely
metrizable space, e.g. a compact metric space, the countable intersection of  dense open sets is dense and so the countable intersection of
dense $G_{\d}$ sets is dense. By contrast, the two sets which are merely dense may have an empty intersection, e.g. the rationals and the irrationals in
the real line. For a subset $A$ of a completely metrizable space $X$ we say that $A$ is \emph{residual} when it contains a  $G_{\d}$ set dense in $X$.
We then say that the members of $A$ are \emph{generic} in $X$. The complement of a residual set is called \emph{meagre}. Thus, the generic real number is irrational.

In a measure space there is an alternate notion of genericity. A subset is said to have \emph{full measure} when its complement has measure zero. The countable
intersection of subsets of full measure again has full measure. If the measure space has a topology such that every nonempty open subset is measurable with positive
measure, then the measure is called \emph{full} and a subset with full measure is dense.

Unfortunately, these two notions conflict. For example, on the unit interval $I$ Lebesgue measure is full. While the classical Cantor set has measure zero,
there exists for every positive integer $n$, a Cantor subset $C_n$ of $I$ with measure $1 - \frac{1}{n}$, obtained by removing ''middle third divided by $n$''
intervals. Since no $C_n$ contains an interval of positive length, the complement of each is dense.  The union of the $C_n$'s is then a meagre subset with
full measure.

A classical, and less artificial, example is the set of \emph{normal numbers} in $I$. A number is normal when every sequence of $n$ digits occurs with
frequency $1/10^n$ in its decimal expansion.  That the set of normal numbers has full measure was proved by Emile Borel in 1909 \cite{Bo}.
However, it is a meagre subset of $I$, see, e.g. \cite{O}.

In a recent paper, \cite{Bob}, Bobok et al. consider the space of continuous maps on the unit interval which preserve Lebesgue measure. These maps are far
from invertible. In fact, the only homeomorphisms of the interval
which preserve Lebesgue measure are the identity $id_I$ and $1 - id_I$. Among other properties the authors
show that for  a generic such map $T$, the point-inverse $T^{-1}(t)$ contains a perfect subset (and so is uncountable) for every
$t \in (0,1)$, i.e. for all points except the end-points.
 \vspace{.5cm}

\section{ \textbf{Cantor Spaces}}\vspace{.25cm}

When you first meet the classical Cantor set obtained by removing successively the ``middle thirds'', the points you see are the endpoints of the
open intervals which were removed. There are countably many such points.  One then observes that the points of the Cantor set are those points of the
unit interval who have a ternary expansion consisting only of $0'$s and $2$'s, no $1$'s, and note that this expansion is unique. This shows that the set is
uncountable.

Thus, in addition to the obvious endpoints, there are uncountably many hidden points.  The distinction between these two sorts of points makes it hard to
credit that the Cantor set is homogeneous, i.e. for any two points of the Cantor set there exists a homeomorphism of the set taking one point to the other. However,
the distinction between the two sorts of points is an artifact of the order structure inherited from the interval.  When you consider the ternary expansions
and replace the $2$'s by $1$'s you map the Cantor set  to the countable product $\{0, 1 \}^{\N}$.  With the product
topology this a compact metrizable space and the map is a homeomorphism. On the other hand $\{0, 1 \}$, regarded as the group of
integers mod $2$, is an additive group and with the product topology $\{0, 1 \}^{\N}$ becomes a compact topological group
via coordinate-wise addition. Such a group is always homogeneous.
The map $x \mapsto x + (b - a)$ is a homeomorphism taking $a$ to $b$. \vspace{.25cm}

\begin{df}\label{df2.01} A \emph{Cantor space} $X$ is a topological space satisfying the following properties.
\begin{itemize}
\item $X$ is compact and metrizable.

\item $X$ is zero-dimensional, i.e. the clopen sets form a basis for the topology.

\item $X$ is perfect, i.e. there are no isolated points, or, equivalently, every non-empty open subset is infinite.
\end{itemize}\end{df}\vspace{.25cm}

Recall that  for a compact, metrizable space all admissible metrics are uniformly equivalent. Taking the union over $n$
of finite covers by open sets of diameter less than $1/n$ we obtain a countable basis and so see that a compact metrizable space is separable.
It then follows that in a
compact, metrizable space the collection of
clopen subsets is countable since each is a finite union of elements of the basis. A compact, metrizable space is zero-dimensional if and only if it is
totally disconnected, i.e. the only connected subsets are singletons. Notice that with the relative topology any nonempty clopen
subset of a Cantor space satisfies the above properties and so is itself a Cantor space.

We will use $C = \{ 0. 1 \}^{\N}$ as our standard Cantor space. It clearly satisfies the above conditions as does the classical middle thirds
Cantor set in the unit interval. An element
$z \in C$ is an infinite string of $0$'s and $1$'s. With $[n] = \{ 1,2, \dots, n \}$, the set $\{ 0, 1 \}^{[n]}$ consists of the $2^n$ $n$-strings,
i.e. strings of length $n$. Given a string $z$ which is infinite or of length at least $m$, we call the restriction $z|[m]$ the initial $m$-string of
$z$. Given an $m$-string $w$ the set $C_w = \{ z \in C : z|[m] = w \}$ is a clopen set and the collection of $C_w$'s as $w$ varies over all finite
strings provides $C$ with a basis of clopen sets.

A \emph{partition} $\P$ of a set $X$ is a  pairwise disjoint collection of nonempty subsets which covers $X$. For partitions $\P_1$ and $\P_2$ of
$X$, $\P_1$ \emph{refines} $\P_2$, written $\P_1 \le \P_2$, when every element of $\P_1$ is contained in a, necessarily unique, element of $\P_2$,
or, equivalently, when every element of $\P_2$ is a union of elements of $\P_1$. A  sequence of partitions $\{ \P_n \}$
is \emph{monotone} when $m \ge n$ implies $\P_m \le \P_n$. By transitivity of the refinement relation it suffices $\P_{n+1}$ refines $\P_n$ for all $n$.
For partitions $\P_1$ and $\P_2$, the \emph{common refinement} is given by
\begin{equation}\label{eq2.01a}
\P_1 \wedge \P_2 = \{ U_1 \cap U_2 : U_1 \in \P_1, U_2 \in \P_2 \} \setminus \{ \emptyset \}.
\end{equation}

 If $X$ is a compact metric space, then $\P$ is a \emph{clopen partition}
when the elements are clopen subsets,  so $\P$ is finite, and the \emph{ mesh} of the partition $\P$,
written $|\P|$, is the maximum of the diameters of the elements of $\P$.

\begin{prop}\label{prop2.02}(\textbf{Uniqueness}) Any two Cantor spaces are homeomorphic. \end{prop}

\begin{proof} As this result is well-known we will just sketch the proof.

Given a Cantor space $X$ we build homeomorphism from $X$ onto our standard space $C$
by constructing inductively a monotone sequence $\{\P_n \}$ of clopen partitions on $X$,
an increasing sequence $\{ m_n \} $ of positive integers and a labelling bijection
$w : \P_n \to \{ 0, 1\}^{[m_n]}$ such that if $V \in \P_{n+1}$ is contained in $ U \in \P_n$, then the label of $U$ is an initial string of the label of $V$. Thus,
\begin{equation}\label{eq2.01}
V \subset U \qquad \Longleftrightarrow \qquad w(U) = w(V)|[m_n].
\end{equation}
In addition,
\begin{equation}\label{eq2.02}
|\P_n| \le 1/n.
\end{equation}

For the inductive step, we can choose $k_n$ large enough that every $U \in \P_n$ can be partitioned by $2^{k_n}$ clopen sets each of diameter
less than $1/(n+1)$. Put these together to define $\P_{n+1}$.
 Let $m_{n+1} = m_n + k_n$. Label the elements of $\P_{n+1}$ which are contained in $U$
using those $m_{n+1}$-strings whose initial $m_n$-strings equal $w(U)$. %Thus, (\ref{eq2.01}) holds.

For each point $x \in X$, there is a unique sequence $\{ U_n \in \P_n \}$ with $x \in U_n$. Because the sequence
$\{ \P_n \}$ is monotone, $U_{n+1} \subset U_n$. Thus, there exists a unique $w(x) \in C$ such that $w(x)|[m_n] = w(U_n)$.
On the other hand, if $z \in C$, there is a sequence $\{ U_n \in \P_n \}$ with $z|[m_n] = w(U_n)$. From (\ref{eq2.01})
it follows that $U_{n+1} \subset U_n$. The decreasing sequence $\{ U_n \}$ of nonempty compacta has a nonempty intersection
and (\ref{eq2.02}) implies that the intersection is a singleton $x$.  Clearly, $w(x) = z$. Thus, the map $x \mapsto w(x)$ is a bijection
from $X$ to $C$. From
 \begin{equation}\label{eq2.03}
 x \in U_n \qquad \Longleftrightarrow \qquad w(x)|[m_n] = w(U_n),
 \end{equation}
 it follows that the map is a homeomorphism.

 \end{proof}\vspace{.25cm}

 For our next result, assume that $T : X \to Y$ is a set map, if $A \subset X, B \subset Y$ with $U = X \setminus A, V = Y \setminus  B$, then,
 \begin{equation}\label{eq2.04}\begin{split}
T(A) \cap B = \emptyset  \quad \Longleftrightarrow \quad  T(A) \subset V \quad \Longleftrightarrow \quad A \subset T^{-1}(V)\\
 \Longleftrightarrow \quad A \cap T^{-1}(B) = \emptyset \quad \Longleftrightarrow \quad T^{-1}(B) \subset U.\hspace{1cm}
 \end{split}\end{equation}

 For compact metrizable spaces $X$ and $Y$ we let $\CC(X,Y)$ be the set of continuous maps from $X$ to $Y$. We denote by $\CC_{sur}(X,Y)$ and by
 $\CC_{iso}(X,Y)$  the subsets of surjective maps and bijective maps, respectively. By compactness, a bijective continuous map is a homeomorphism.
 Given a metric $d$ on $Y$, we use the sup metric $d$ on $\CC(X,Y)$ given by
 \begin{equation}\label{eq2.05}
 d(T_1, T_2) = \sup \{ d(T_1(x),T_2(x)) : x \in X \}.
 \end{equation}
 The associated topology is that of uniform convergence and with this metric $\CC(X,Y)$ is a complete metric space.
 With $U$ and $V$ varying over countable bases in $X$ and $Y$, respectively,
 the collection \\
 $\{ T : T(\ol{U}) \subset V \}$
 provides a countable subbase for $\CC(X,Y)$.  Thus, $\CC(X,Y)$ is separable as well.\vspace{.25cm}

\begin{prop}\label{prop2.03}($\mathbf{G_{\d}}$) For compact metrizable spaces $X$ and $Y,$ \\ $\CC_{sur}(X,Y)$ is a closed subset of
$\CC(X,Y)$ and $\CC_{iso}(X,Y)$ is a $G_{\delta}$ subset of $\CC(X,Y)$. \end{prop}

\begin{proof} If $A, B$ are closed subsets of $X$ and $Y$ respectively with open complements $U$ and $V$ then the compact sets
 $T(A)$ and $B$ are  a positive distance apart if they are disjoint.  So $\{ T\in \CC(X,Y) : T(A) \cap B = \emptyset \}$ is
 an open subset of $\CC(X,Y)$.

 For any proper open subset $V$ of Y, it follows from (\ref{eq2.04})that $\{ T \in \CC : T(X) \subset V \}$ is an open set and taking the union over
 all such proper open sets $V$, we obtain the open set $\CC \setminus \CC_{sur}$.  Hence, $\CC_{sur}$ is closed.

 Clearly, $T \mapsto T \times T$ is a continuous map from $\CC(X,Y)$ to $\CC(X \times X,\\ Y \times Y)$. With the metric $d$ on $X$, the
 set $V_{1/n} = \{ (x_1,x_2) : d(x_1,x_2) < 1/n \}$ is open in $X \times X$. With $\Delta_Y   = \{ (y,y)  \}$, the closed diagonal in $Y,$ it
 follows from (\ref{eq2.04})that $\{ T \in \CC :(T \times T)^{-1}(\Delta_Y) \subset V_{1/n} \}$ is open. Intersecting over $n$ we see that the
 set of injective maps is $G_{\d}$. Since a closed set in a metric space is always $G_{\d}$, it follows that $\CC_{iso}$ is $G_{\d}$.

 \end{proof}\vspace{.25cm}

\begin{prop}\label{prop2.04}(\textbf{Density}) Let $T : X \to Y$ be a surjective, continuous map with $X$ and $Y$ Cantor spaces.  For $d$ an admissible
metric on $Y$ and $\ep > 0$, there exists a homeomorphism $T_1$ of $X$ onto $Y$ such that
 \begin{equation}\label{eq2.06}
 x \in X \qquad \Longrightarrow \qquad d(T(x),T_1(x)) \le \ep.
 \end{equation}. \end{prop}

 \begin{proof} Let $\P = \{U_1, \dots, U_n \}$ be a clopen partition of $Y$ with mesh at most $\ep$. Since $T$ is surjective and
 each $U_i$ is nonempty, it follows that each $T^{-1}(U_i)$ is a nonempty clopen subset of $X$ and so
 $ \{ T^{-1}(U_1), \dots, T^{-1}(U_n) \}$ is a clopen partition of $X$.
 Each $U_i$ and each $T^{-1}(U_i)$ is a Cantor space and so  Proposition \ref{prop2.02} implies we can
 choose a homeomorphism $\hat T_i $ from $ T^{-1}(U_i)$ onto $U_i$. Concatenating we obtain the homeomorphism $T_1$ from $X$ onto $Y$.
 Clearly, for all $x \in X$
 \begin{equation}\label{eq2.07}
 T(x) \in U_i \qquad \Longleftrightarrow \qquad T_1(x) \in U_i.
 \end{equation}
 Since the partition $\P$ has mesh at most $\ep$, (\ref{eq2.06}) follows.

 \end{proof}\vspace{.25cm}

 From Propositions \ref{prop2.03} and \ref{prop2.04} we immediately obtain:\vspace{.25cm}

 \begin{theo}\label{theo2.05} If $X$ and $Y$ are Cantor spaces, then the set of homeomorphisms $\CC_{iso}(X,Y)$ is a $G_{\d}$ subset of
 the completely metrizable space $\CC(X,Y)$ and it is dense in the closed subset $\CC_{sur}(X,Y)$ of surjective maps.
 Thus, the generic surjective map is invertible.\end{theo} \vspace{.5cm}

\textbf{ Notes:} The uniqueness of Cantor space is a classical result, first proved by L. E. J. Brouwer \cite{B} in 1909. Metrizability is not redundant.
There exist many distinct examples of compact, separable, first countable, zero-dimensional, perfect spaces.  The \emph{Sorgenfrey Double Arrow} construction yields
a linearly ordered topological space with these properties. In it every point is an endpoint and there are uncountably many clopen subsets. For a discussion
of such spaces and the connected linearly ordered spaces in which they live, see \cite{AH02} and its sequel \cite{AH21}.

For the density results it suffices to consider the case with $Y = X$. As part of a study of the dynamics of surjective self-maps of the Cantor space in
\cite{A16} it is proved that the group of invertible self-maps is a dense $G_{\d}$ subset. In fact, it then
follows from the work of Kechris and Rosenthal \cite{KR} that there
is a single conjugacy class which is a dense $G_{\d}$ subset.  For an explicit description of a member of this class, see \cite{AGW}.
 \vspace{.5cm}

 \section{\textbf{Measure Spaces}}\vspace{.25cm}

 A \emph{measure space} is a triple $(X, \B, \mu)$ with $\B$ a $\s$-algebra of subsets of $X$ and $\mu$ a
 finite or $\s$-finite measure   on the elements of $\B$. We assume $\mu$ is non-trivial, i.e. $\mu(X) > 0$.
 It is \emph{normalized} when $\mu(X) = 1$. An element $A \in \B$ has \emph{full measure} when $X \setminus A$ has
 measure zero. In the finite case, this is equivalent to $\mu(A) = \mu(X)$. In general, we will write $(X,\mu)$
for $(X, \B, \mu)$ with the $\s$-algebra $\B$ understood by context.

 An \emph{atom} $A$ is an element of $\B$ with $\mu(A) > 0$ such that
 \begin{equation}\label{eq3.01}
 B \in \B \  \text{with} \ B \subset A \qquad \Longrightarrow \qquad \mu(B) = \mu(A) \ \text{or} \ \mu(B) = 0.
 \end{equation}
 Since an infinite measure is assumed to be $\s$-finite, it follows that even in that case an atom has finite measure.
 The measure is \emph{nonatomic} when it has no atoms.

 Let $\B'$ be the set of those elements of $\B$ with finite measure. On $\B'$, the measure induces a pseudo-metric $d$ given by
  \begin{equation}\label{eq3.02}
  d(A,B) = \mu(A + B)
  \end{equation}
  where $A + B$ is the Boolean sum $(A \cup B) \setminus (A \cap B)$. The set operations: union, intersection and complementation,
  are continuous with respect this metric. If $\{ A_n \}$ is a decreasing sequence of elements of $\B'$
  then the sequence converges to  $\bigcap_n A_n$. If $\{ A_n \}$ is a increasing sequence of elements of $\B'$,
  then it converges to $\bigcup_n A_n$ if and only if the measure sequence $\{ \mu(A_n) \}$ is bounded.

  The associated \emph{measure algebra} $\M_X$ is the associated metric space. That is, the elements are the equivalence classes of sets in $\B'$ which
  differ by a set of measure zero. The set operations and the measure are well-defined using representatives from the classes.
 % Using either $\B$ or its   completion yields the same measure algebra.
  We denote by $0$ the class of sets of measure zero. In general, we will write $A$ both for an element of
  $\M_X$ and for any element of $\B'$ which represents it.

  The measure space is separable when there is a countable subset $\B_0 \subset \B'$, which we may take to be a subalgebra, such that every
  element of $\B'$ can be arbitrarily closely approximated by an element of $\B_0$.  That is, given $A \in \B'$ and $\ep > 0$, there exists
 $B \in \B_0$ such that $d(A,B) < \ep$.
  This exactly says that the classes of $\B_0$ form a dense subset of the metric space $\M_X$. Thus, the
  measure space is separable exactly when the measure algebra is a separable metric space.
  \vspace{.25cm}

  \begin{prop}\label{prop3.01} For a measure space, the associated measure algebra is a complete metric space. \end{prop}

  \begin{proof}  Let $\{ A_n \in \B \}$ represent a Cauchy sequence in $\M$. The sequence of
  characteristic functions $\{ 1_{A_n} \}$ is clearly fundamental in
  measure in the sense of \cite{H50} Section 22 of Chapter IV. By Theorem E of that section the sequence $\{ 1_{A_n} \}$ converges in measure to a
  measurable function $k$ and a subsequence converges almost uniformly, and so pointwise a.e., to the function $k$.  It then follows that $k$ is
  the characteristic function of some set $A$ and that $\{ d(A_n,A) \}$ has limit zero.

  \end{proof}  \vspace{.25cm}

  For $A \in \B_X \ \ \B_A = \{ A_1 \in \B_X : A_1 \subset A \} = \{ A \cap A_1 : A_1 \in \B_X \}$.  Then $(A,  \mu|A)$ is a measure
  subspace of $(X, \mu)$ where $\mu|A(A_1) = \mu(A_1)$ for $A_1 \in \B_A$. The associated measure algebra $\M_A$ is a subspace of $\M_X$.

  For $(X, \mu)$ a \emph{measure space partition}, $\P$, is a finite or  infinite sequence $ \{ A_k \in \B' \}$ such that
  $\mu(A_i \cap A_j) = 0$ when $i \not= j$ and $\bigcup_i A_i $ is a set of full measure. These represent a
 \emph{ measure algebra partition} $ \{ A_i \in \M_X \}$ such that
  $A_i \cap A_j = 0$ when $i \not= j$ and $A = \bigcup_i A \cap A_i$ for all $A \in \M_X$. When $\{ A_i \}$ is a sequence in $\M_X$ such that
  $A = \bigcup_i A \cap A_i$ for all $A \in \M_X$ we will say that $ \{  A_i \} $ is a full sequence in $\M_X$. In the finite case, this just says
   that  $ X = \bigcup_i \ A_i$.
  In either case, the $\mu$-mesh $|\P|$ is the maximum of the measures $\mu(A_i)$.

  For measure spaces $(X,  \mu)$ and $(Y, \nu)$ a mapping $T : X \to Y$ is a \emph{measure space map} when it is a measurable function which
  preserves measure, i.e.
    \begin{equation}\label{eq3.03}
  B \in \B_Y \quad \Longrightarrow \quad T^{-1}(B) \in \B_X \quad \text{with} \ \ \mu(T^{-1}(B)) = \nu(B).
  \end{equation}
  Thus, $T^{-1}(\B'_Y) \subset \B'_X.$

  The map $T$ is a \emph{measure space isomorphism} when there exists a measure space map $\hat T : (Y, \nu) \to (X, \mu)$ such that
  $\hat T \circ T = id_X$ a.e. and $T \circ \hat T = id_Y$ a.e. \vspace{.25cm}

  \begin{prop}\label{prop3.02a} Let $T : (X,  \mu) \to (Y, \nu)$ be a measure space map.

  (a)  If $T : (X,  \mu) \to (Y,  \nu)$ is a measure space isomorphism with inverse $\hat T$ then there exist
  subsets $X_1 \subset X$ and $Y_1 \subset Y$ of full measure such that the restriction
  $T|X_1$ is a bijection onto $Y_1$ with inverse the restriction $\hat T|Y_1$.

  (b) Conversely, if there exist subsets $X_1 \subset X$ and $Y_1 \subset Y$ of full measure
  such that the restriction $T|X_1$ is a bijection onto $Y_1$ with a measurable inverse, then $T$ is a measure space isomorphism. \end{prop}

  \begin{proof} (a) Let $X_0 \subset X$ and $Y_0 \subset Y$ be subsets of full measure such that $\hat T \circ T $ is the identity on $X_0$ and
  $T \circ \hat T $ is the identity on $Y_0$. Let $X_1 = X_0 \cap T^{-1}(Y_0)$ and $Y_1 = Y_0 \cap \hat T^{-1}(X_0)$. If $x \in X_1$ then
  $T(x) \in Y_0$ and since $x \in X_0$, $\hat T(T(x)) = x$. Hence, $T(x) \in \hat T^{-1}(X_0)$.  That is, $T(X_1) \subset Y_1$ and similarly,
  $\hat T(Y_1) \subset X_1$. The maps $T : X_1 \to Y_1$ and $\hat T : Y_1 \to X_1$ are clearly inverses.

  (b) Define $\hat T$ to be $(T|X_1)^{-1}$ on $Y_1$ and map $Y \setminus Y_1$ to a point $x$ of $X_1$. Since $\hat T$ maps into $X_1$, we have for
  $A \in \B_X$, $\hat T^{-1}(A) = T(A \cap X_1) $ if $x \not\in A$ or $T(A \cap X_1) \cup (Y \setminus Y_0)$ if $x \in A$. These are measurable
  because $(T|X_1)^{-1}$ is assumed measurable and  $T(A \cap X_1) = ((T|X_1)^{-1})^{-1}(A \cap X_1)$. Since $A \cap X_1 = T^{-1}(T(A \cap X_1)) \cap X_1$
  we have
  \begin{equation}\label{eq3.03a}
  \mu(A) = \mu(A \cap X_1) = \mu(T^{-1}(T(A \cap X_1)) = \nu(T(A \cap X_1)) = \nu(\hat T^{-1}(A)).
  \end{equation}
  Hence, $\hat T : (Y,  \nu) \to (X,  \mu)$ is a measure space map. Since $\hat T \circ T$ is the identity on $X_1$ and $T \circ \hat T$ is the
  identity on $Y_1$, $\hat T$ is the inverse of $T$.

  \end{proof}\vspace{.25cm}

  In particular, if $A_1 \subset A \in \B$ with $\mu(A \setminus A_1) =  0$, then the inclusion map from $A_1$ to $A$ is a measure space isomorphism.

  Note that a measurable bijection need not have a measurable inverse. For example, if $\B_0$ is a $\s$-algebra  of subsets of $X$
  which is properly contained in the $\s$-algebra $\B$, then the identity map on $X$ is measurable from $\B$ to $\B_0$ but the inverse is not.
 However, if a  bijection preserves measures and has a measurable inverse, then the inverse clearly preserves measures as well.

    If $T : (X,  \mu) \to (Y, \nu)$ is a measure space map, then the map from $\B'_Y$ to $\B'_X$ given by $B \mapsto T^{-1}(B)$
    induces a map $T^* : \M_Y \to \M_X$, which is an example of a measure algebra map. \vspace{.25cm}

    \begin{df}\label{df3.02} For measure algebras $\M_X$ and $\M_Y$ with $\mu(X) = \nu(Y)$ a mapping $S : \M_Y \to \M_X$ is
    a \emph{measure algebra map} when it satisfies:
    \begin{itemize}

    \item[(i)] $S(B_1 \cup B_2) = S(B_1) \cup S(B_2)$ and $S(B_1 \setminus B_2) = S(B_1) \setminus S(B_2)$ for all $B_1, B_2 \in \M_Y$.

     \item[(ii)]$\mu(S(B)) = \nu(B)$ for all $B \in \M_Y$.

    \item[(iii)]If $\{ B_i \}$ is a full sequence in $\M_Y$, then $\{ S(B_i) \}$ is a full sequence in $\M_X$.
    \end{itemize}
    \end{df}\vspace{.25cm}
    It follows that $S(0) = 0$ and in the finite case $S(Y) = X$. Thus, in the finite case  condition (iii) is redundant. Since
    $B_1 \setminus B_2 = (B_1 \cup B_2) \setminus B_2$, it suffices in (i) that differences be preserved when $B_2 \subset B_1$.

    Since $B_1 + B_2 = (B_1 \setminus B_2) \cup (B_2 \setminus B_1)$ and
    $B_1 \cap B_2 = (B_1 \cup B_2) \setminus (B_1 + B_2)$, it follows that $S$  preserves
    Boolean sums and intersections. Hence, $S$ preserves the metrics as well. That is, $S : \M_Y \to \M_X$ is a
    metric space isometry.  In particular, a measure algebra map is always injective.  When it is surjective as well, then the inverse mapping
    is clearly a measure algebra map from $\M_X$ to $\M_Y$ and so $S$ is then a measure algebra isomorphism.

    The arguments that follow will all be at the measure algebra level.  At the end of the section, we will describe the results needed to lift
    the arguments to measure space conclusions.

    Our standard example of a normalized, finite, separable measure space will be the unit interval $I = [0,1]$ with Lebesgue measure $\lm$ on $\B_I$
    the Borel sets in $I$ and measure algebra $\M_I$. The measure space $(I,\lm)$ is separable and nonatomic.

  \begin{prop}\label{prop3.03}(\textbf{Uniqueness}) Any two normalized, separable,\\ nonatomic measure algebras are isomorphic. \end{prop}

  \begin{proof} It suffices to show that a normalized, separable, nonatomic  measure algebra $\M_X$ is isomorphic to $\M_I$. Let $\{E_1, E_2, \dots \}$
  be a dense sequence in $\M_X$. We construct a monotone sequence of measure algebra partitions $\{ \P_n \}$ and maps $S : \P_n \to \B_I$
  such that
  \begin{itemize}
  \item The $\mu$-mesh $|\P_n|$ is at most $1/n$ and $E_n$ is a union of elements of $\P_n$.
  \item The set $S(\P_n)  = \{ S(A) : A \in \P_n \}$ is a partition of $I$ by closed intervals with $\lm(S(A)) = \mu(A)$.
  \item If $A \in \P_n, B \in \P_{n+1}$ with $B \subset A$, then $S(B) \subset S(A)$.
  \end{itemize}
  Thus,  $ \{ S(\P_n)  \} $ is a monotone sequence of partitions of $\M_I$ with $|S(\P_n)| = |\P_n|$.

  We begin the inductive construction with $\P_1 = \{ E_1, X \setminus E_1 \}$. Let $S(E_1)$ be the interval $[0,\mu(E_1)]$ and let
  $S(X \setminus E_1)$ be the interval $[\mu(E_1),1]$. Since $\M_X$ is normalized, $\mu(X \setminus E_1) = 1 - \mu(E_1)$.

  Now assume that $\P_k$ and $S : \P_k \to \B_I$ have been constructed for $k \le n$. Because $\M_X$ is nonatomic, we can choose a
  refinement $\P_{n+1}$ of the common refinement $\P_n \wedge \{ E_{n+1}, X \setminus E_{n+1} \}$ such that $|\P_{n+1}| \le 1/(n+1)$.
  Given $A \in \P_n$,  $\P_{n+1}$ contains a partition $\{ B_1, \dots,B_k \}$ of $A$. Now $S(A)$ is an interval $[a,b]$ of length $\lm(S(A)) = \mu(A)$.
  Choose $S(B_1) = [a,a+\mu(B_1)], S(B_2) = [a+\mu(B_1),a+\mu(B_1)+\mu(B_2)],$ etc to obtain a partition of $[a,b]$ by closed intervals. Letting
  $A$ vary over $\P_n$ we obtain the partition $\P_{n+1}$.   This completes the inductive construction.

  Let $\B_0$ be the subalgebra of $\B'_X$ generated by the representatives of $\bigcup_n \{ \P_n \}$ and let $\M_0$ be the associated
  subset of $\M_X$. If
  $A \in \M_0$ then for some $n$ $A$ is a union of elements of $\P_n$
  and we let $S(A)$ be the union of the corresponding intervals.  Since the sequences of partitions are monotone, the definition of $S(A)$ is
  independent of the choice of large enough $n$. Thus, we obtain a measure algebra map $S : \M_0 \to \M_I$. As $S$ is an isometry and $\M_0$ is
  dense in $\M_X$, $S$ extends to an isometry from $\M_X$ to the complete metric space $\M_I$. From continuity of the algebraic operations it follows
  that $S$ is a measure algebra map on all $\M_X$. Finally, since the mesh of the $S(\P_n)$ tend to zero, every interval in $I$ can be arbitrary closely
  approximated by a union of intervals in $\bigcup_n \{ S(\P_n) \}$. Thus, $S(\M_0)$ is dense in $\M_I$ and so $S : \M_X \to \M_I$ is surjective and is
  thus an isomorphism.

  \end{proof} \vspace{.25cm}

    Our standard example of an infinite, separable, nonatomic measure space will be
    the  real line $\R$ with Lebesgue measure $\lm$ on $\B_{\R}$
    the Borel sets in $\R$ and measure algebra $\M_{\R}$. The measure space $(\R,\lm)$ is separable and nonatomic. \vspace{.25cm}

  \begin{prop}\label{prop3.03y}(\textbf{Uniqueness}) Any two infinite, separable,\\ nonatomic measure algebras are isomorphic. \end{prop}

  \begin{proof} We construct an isomorphism of  $\M_X$ with $\M_{\R}$. Let
  $\P = \{ A_k \in \B' : k \in \Z \}$ be a partition of $X$ by sets of finite, positive measure with $a_k = \mu(A_k)$. Since $\mu(X) = \infty, \
  \sum_k \ a_k = \infty.$ We choose the numbering so that $\sum_{ k > 0 } \ a_k$ and $\sum_{k < 0} \ a_k$ are both infinite.
  Let $J_0 = [0,a_0)$ and for $k > 0$ let $J_k = [a_0 + \dots + a_{k-1}, a_0 + \dots + a_k)$.
  Let $J_{-1} = [a_{-1},0)$ and for $k < -1$ let $J_k = [a_k + \dots a_{-1}, a_{k+1} + \dots a_{-1})$. By normalizing,
  we can apply Proposition \ref{prop3.03} to obtain a measure space isomorphism $S_k : (A_k,\mu|A_k) \to (J_k,\lm|A_k)$
  and define for $A \in \M_X, B \in \M_I$
  \begin{equation}\label{eq3.06a}\begin{split}
 S(A) = \bigcup_k \ S_k(A \cap A_k), \hspace{1cm}\\
 S^{-1}(B) = \bigcup_k \ (S_k)^{-1}(B \cap S(A_k)).
 \end{split}\end{equation}
 to get the required isomorphism.

  \end{proof}\vspace{.25cm}

  For measure algebras $\M_Y$ and $\M_X$ let $\S(\M_Y,\M_X)$ be the set of measure algebra maps from $\M_Y$ to $\M_X$ and let $\S_{iso}(\M_Y,\M_X)$
  be the subset of isomorphisms. On $\S$ we use the coarse topology of pointwise convergence. That is, a net $\{ S_i \}$ converges to $S$ when
  $\{ S_i(A) \}$ converges to $S(A)$ for all $A \in \M_Y$. Since the elements of $\S$ are isometries, it suffices to check this on a dense subset
  of $\M_Y$. Thus, in the separable case with  $\{ E_n : n \in \N \}$  a dense sequence in $\M_Y$,  the topology is given by the metric
     \begin{equation}\label{eq3.04}
     d(S_1, S_2) = \max_n \ \frac{1}{n} \min (d(S_1(E_n),S_2(E_n)), 1)
     \end{equation}
     Notice that if $\{ S_n \}$ is a Cauchy sequence in $\S$, then for each $A \in \M_Y$, $\{ S_n(A) \}$ is a Cauchy sequence in the complete space
     $\M_X$ and so it converges to an element which we label $S(A)$.  It follows from continuity of the measure  and of the algebraic operations
     that $S$ is a measure algebra map and so is the limit in $\S$ of the sequence. Thus, if $\M_X$ and $\M_Y$ are separable, then $\S(\M_X,\M_Y)$ is
     a completely metrizable space.

     Furthermore, the countable product of separable metric  spaces is a separable metrizable space with the product topology. From the metric
     \ref{eq3.04} we see that the map $S \mapsto \{ S(E_i) \}$ is an embedding of $\S(\M_Y, \M_X)$ into the product space $\M_X^{\{ E_i \}}$.
     Thus, $\S(\M_Y, \M_X)$ is a completely metrizable, separable space.
     \vspace{.25cm}

\begin{prop}\label{prop3.04}($\mathbf{G_{\d}}$) If $\M_Y$ and $\M_X$ are separable, nonatomic measure algebras, then the set of
isomorphisms $\S_{iso}(\M_Y.M_X)$ is a $G_{\d}$ subset of $\S(\M_Y.M_X)$.\end{prop}

\begin{proof} Let $\{E_i \} $ be a dense sequence in $\M_X$. If $S \in \S$, then because it is an isometry, the range $S(\M_Y)$ is a closed subset
of $\M_X$. Since $S$ is not an isomorphism exactly when it is not surjective, this occurs when some $E_i$ is at a positive distance from the range of $S$.
Define for  positive integers $i, k$
\begin{equation}\label{eq3.05}
\S_{ik} = \{ S : d(E_i, S(B)) \ge 1/k \ \text{for all} \ B \in \M_Y \}.
\end{equation}
This is a closed set and $\S \setminus \S_{iso}$ is the union of the countable collection obtained by varying $i$ and $k$. Since
$\S \setminus \S_{iso}$ is an $F_{\s}$ set, the complement $\S_{iso}$ is a $G_{\d}$.

\end{proof}\vspace{.25cm}

\begin{prop}\label{prop3.05}(\textbf{Density}) Let $S : \M_Y \to \M_X$ be a measure algebra map with $\M_X$ and $\M_Y$  separable,
nonatomic measure spaces. If $\{ E_1, E_2, \dots \}$ is a dense sequence in $\M_Y$, then for any positive integer $n$, there exists
a measure algebra isomorphism $S_n : \M_Y  \to \M_X$ such that $S(E_i) = S_n(E_i)$ for $i = 1, \dots, n$. The sequence
$\{ S_n \}$ converges to $S$ in $\S(\M_Y,\M_X)$.  \end{prop}

\begin{proof} In the infinite case we begin with an infinite partition $\P_0$ of $\M_Y$, i.e. a partition of $Y$ by elements of $\B'_Y$.
In the finite case, let $\P_0 = \{ Y \}$  Let $\P = \{ B_i\} $ be the common refinement partition of $\P_0$ with
 $\bigwedge_{i = 1}^n \ \{E_i, Y \setminus E_i \}$ and let
$S(\P) = \{ S(B_i) \}$ be the image partition of $\M_X$.  Note that $\P_0$ and hence also $\P$ are full sequences in $\M_Y$ and
so $S(\P)$ is a full sequence in $\M_X$. For each $k$, $\M_{B_k} \subset \M_Y$
and $\M_{S(B_k)} \subset \M_X$ are finite, nonatomic
measure algebras. Each is separable as any subspace of a separable metric space is separable.

We can normalize by dividing the measures
by $\nu(B_k) = \mu(S(B_k))$ and then apply Proposition \ref{prop3.03} to get an isomorphism $\hat S_k : \M_{B_k} \to \M_{S(B_k)}$. We obtain
the isomorphism $S_n$ by defining for $B \in \M_Y$ and $A \in \M_X$:
 \begin{equation}\label{eq3.06}\begin{split}
 S_n(B) = \bigcup_k \ \hat S_k(B \cap B_k), \hspace{1cm}\\
 S_n^{-1}(A) = \bigcup_k \ (\hat S_k)^{-1}(A \cap S(B_k)).
 \end{split}\end{equation}
 Since $S_n(B_k) = S(B_k)$ for all $k$ and each $E_i$ with $i \le n$ is a union of the $B_k$'s it follows that $S_n(E_i) = S(E_i)$.

 For each $i$, the sequence $\{ S_n(E_i) \}$ is  constant at $S(E_i)$ once $n \ge i$. Since $\{ E_i \}$ is dense in $\M_Y$, the sequence
 $\{ S_n \}$ converges to $S$.

 \end{proof}
\vspace{.25cm}

 From Propositions \ref{prop3.04} and \ref{prop3.05} we  obtain: \vspace{.25cm}

 \begin{theo}\label{theo3.06}  If $\M_Y$ and $\M_X$ are separable, nonatomic measure algebras, then the set $\S_{iso}(\M_Y,\M_X)$ of
 measure algebra isomorphisms is a dense, $G_{\d}$ subset of the
 completely metrizable space $\S(\M_Y,\M_X)$ of measure algebra mappings. Thus, the generic measure algebra map is invertible.\end{theo}
\vspace{.25cm}

Our lifting results will concern measure space maps between nonatomic measure spaces of Borel measures on Polish spaces.  A topological space is
a \emph{Polish space} when it is separable and admits a complete metric.  A subset $A$ of a Polish space $X$ has a Polish relative topology if and
only if $A$ is a $G_{\d}$ subset of $X$. When $(X,\mu)$ is a measure space with $X$ a Polish space and $\B_X$ the Borel sets  we will call
$(X,\mu)$ a \emph{Polish measure space}.

Our results are taken from Royden \cite{R} Chapter 15, \emph{Mappings of Measure Spaces}.

For a Borel measure $\mu$ on a separable metrizable space $X$, the support of $\mu$, written $supp(\mu)$, is the complement of the union of all
open sets of measure zero. Since we may take the union just over the sets in a countable base, we see that the union has measure zero and so
the closed set $supp(\mu)$ has full measure and every open set which meets the support has positive measure. The measure on $X$ is called
\emph{full} when $supp(\mu) = X$ and so when every nonempty open subset has positive measure.

Notice that if $A$ is an atom for $\mu$, then the support of the restriction $\mu|A$ must be a singleton. Thus, $\mu$ is nonatomic exactly when
each singleton has measure zero, or, equivalently, all countable subsets have measure zero. So if $\mu$ is a nonatomic measure on a separable, metrizable space
$X$, then $X$ is uncountable. If $\mu$ is full, then there are no isolated points and so $X$ is perfect.

If a measure on a separable metric space is nonatomic, then each point has an open neighborhood with finite measure. A
separable metric space is Lindel\"{o}f and so the measure is $\s$-finite.

We first observe that the functor from measure space maps to measure algebra maps is injective.\vspace{.25cm}

\begin{prop}\label{prop3.07} Let $T_1, T_2 : (X,\mu) \to (Y, \nu)$ be measure space maps with $\nu$ a Borel measure on
the separable metric space $Y$. If $T_1^* = T_2^* : \M_Y \to \M_X$, then $T_1 = T_2$ a.e. \end{prop}

\begin{proof} The product map $T_1 \times T_2 : X \to Y \times Y$ is measurable because the Borel subsets of $Y \times Y$ are generated by
the sets $B_1 \times B_2$ with $B_1, B_2 \in \B_Y$. Hence, $T_1 = T_2$ a.e. if and only if the set $\{ x : T_1(x) \not= T_2(x) \}$ has measure zero.
If this is not true, then for some positive integer $k$, the set $A =  \{ x : d(T_1(x),T_2(x)) > 1/k \}$ has positive measure.

Let $\P $ be a partition of $Y$ by measurable sets each of metric diameter at most $1/k$. Again we use the
Lindel\"{o}f property. For some $B \in \P$, the set $A \cap T_1^{-1}(B)$ has
positive measure. If $x \in A \cap T_1^{-1}(B)$ then $T_1(x) \in B$ and $d(T_2(x),T_1(x)) > 1/k$ implies $T_2(x) \not\in B$. That is,
$T_2^{-1}(B) $ is disjoint from $A \cap T_1^{-1}(B)$ and so $T_1^*(B) \not= T_2^*(B)$.

\end{proof}\vspace{.25cm}

For the following deep result we refer to \cite{R} Proposition 15.19 (Theorem 15.11 of the second edition).\vspace{.25cm}

 \begin{theo}\label{theo3.08} Let $(X,\mu)$ and $(Y,\nu)$ be separable measure spaces with $X$ and $Y$
 uncountable and $(Y,\nu)$  Polish. If $S : \M_Y \to \M_X$ is a measure algebra map, then there exists $T : (X,\mu) \to(Y,\nu)$ a measure
 space map with $T^* = S$.
 \end{theo}\vspace{.25cm}

 \begin{cor}\label{cor3.09} (a)  Let $(X,\mu)$ and $(Y,\nu)$ be Polish measure spaces with $X$ and $Y$
 uncountable. If $S : \M_Y \to \M_X$ is a measure algebra isomorphism, then there exists $T : (X,\mu) \to(Y,\nu)$ a measure
 space isomorphism with $T^* = S$.

 (b) Let $T : (X,\mu) \to(Y,\nu)$ be a measure
 space map between Polish measure spaces with $X$ and $Y$ uncountable. If $T^*$ is a measure
 algebra isomorphism, then $T$ is a measure space isomorphism.\end{cor}

 \begin{proof} (a) Applying Theorem \ref{theo3.08} to $S$ and $S^{-1}$ we obtain measure space maps $T : (X,\mu) \to(Y,\nu)$
 and $\hat T : (Y,\nu) \to(X,\mu)$ such that $T^* = S$ and $\hat T^* = S^{-1}$. Hence, $(\hat T \circ T)^* = T^* \circ \hat T^* = id_X^*$
 and so by Proposition \ref{prop3.07} $\hat T \circ T = id_X$ a.e. Similarly, $T \circ \hat T = id_Y$ a.e. Thus, $T$ and $\hat T$ are inverse isomorphisms.

 (b) By (a) there is a measure space isomorphism $\hat T$ with $\hat T^* = T^*$. By Proposition \ref{prop3.07}
 $T = \hat T$ a.e.

 \end{proof} \vspace{.25cm}

 By \cite{R} Proposition 15.11 a Borel measure $\mu$ on a metrizable space $X$ is \emph{regular}, i.e. for any Borel set $A$,
 $\mu(A) = \inf \{ \mu(U): U \ $ is open and $ \ A \subset U \}$. It follows that $A$ is contained in a $G_{\d}$ set $A_1$ with
 $\mu(A_1 \setminus A) = 0$. The inclusion of $A$ into $A_1$ thus induces a measure space isomorphism. If $X$ is Polish, then
 $A_1$ is Polish as well. It then follows from Proposition \ref{prop3.03} and Theorem \ref{theo3.08}that if $(X,\mu)$
 are $(Y,\nu)$ are nonatomic Polish measure spaces and $A \in \B_X, B \in \B_Y$
 with $\mu(A) = \nu(B) > 0$, then the measures spaces restricted to $A$ and $B$ are isomorphic. Notice that since the measures are nonatomic, sets of
 positive measure like $A$ and $B$ are uncountable.  \vspace{.25cm}

For Polish measure spaces the map $T \mapsto T^*$ from measure space maps to measure algebra maps is a bijection. So we use the topology from the space
of measure algebra maps to make this a homeomorphism. Thus, a sequence $\{ T_n \}$ of measure space maps converges to $T$ when for every Borel subset $E$ of the range we have
$\mu(T^{-1}_n(E) + T^{-1}(E)) \to 0$.  From Theorem \ref{theo3.06} we have:

 \begin{theo}\label{theo3.10}If $(X,\mu)$ and $(Y,\nu)$ are nonatomic
 Polish measure spaces (and so  with $X$ and $Y$ uncountable) the generic measure space map between them is invertible. \end{theo}\vspace{.5cm}

 \textbf{Notes:} The measure algebra approach is here adapted from Halmos \cite{H50} and \cite{H56}. The isomorphism result Proposition
 \ref{prop3.03} is described there and is attributed by Royden to Caratheodory (\cite{R} Theorem 15.4). Again we may assume that $Y = X$.
 In \cite{E} Eisner proves that the group of invertible measure space self-maps of a standard measure space is a dense $G_{\d}$ subset of the
 space of measure space self-maps. In fact, the proof of Propostion \ref{prop3.04} is essentially hers.
 Instead of the measure algebra isometry $T^*$ on $\M_X$, Eisner uses the \emph{Koopman operator} on $\L^2(X,\mu)$
 given by $T^*(u) = f \circ T,$ for $ f \in \L^2$ as a proxy for the measure space map. Note that $d'(A,B) = \sqrt{\mu(A + B)}$ is a metric on $\M_X$ uniformly
 equivalent to $d$. The map associating to $A \in \M_X$ the characteristic function $1_A \in \L^2(X,\mu)$ is then an isometry from $\M_X$ into $\L^2(X,\mu)$.
It maps $T^*$ to the restriction of the Koopman operator.

 The results leading to the lifting theorem Theorem \ref{theo3.08}, i.e. to  Proposition 15.19 in \cite{R}, is the work of many hands.
 Royden attributes various results to Sikorski, Kuratowski, von Neumann and Halmos.

 In \cite{Bob} the authors show that the generic continuous map of the interval which preserves Lebesgue measure has measure-theoretic entropy zero.
 A measure space map with entropy zero is invertible a.e., see \cite{W} Corollary 4.14.3.
They have thus shown that the generic continuous measure-preserving  map  of the interval, while wildly non-invertible in the topological sense, is nonetheless invertible a.e. They are using the topology of uniform convergence.  In addition,
 within the space of  measure-preserving  maps  of the interval, with the topology we are using, the continuous maps are dense and so every  measure-preserving  map  of the interval
 can be approximated by a continuous measure preserving map which is invertible a.e. This provides an alternate approach to
 the density portion of Theorem \ref{theo3.10} \vspace{.5cm}

  \section{\textbf{Nonsingular Mappings of Measure Spaces}}\vspace{.25cm}

   For finite measure spaces $(X, \mu)$ and $(Y,  \nu)$ (as before, we leave $\B_X$ and $\B_Y$ understood) a mapping $T : X \to Y$ is a \emph{null-preserving map} when it is a measurable function which
which associates sets of measure zero i.e.
    \begin{align}\label{eq3.01x}\begin{split}
  B \in \B_Y \qquad &\Longrightarrow \qquad T^{-1}(B) \in \B_X \\
  \text{and}  \ \ \nu(B) \ = \ 0 \qquad &\Longrightarrow \qquad \mu(T^{-1}(B)) = 0.
  \end{split}\end{align}
  It is a \emph{nonsingular map} when
   \begin{equation}\label{eq3.01xa}
 \hspace{1cm} \nu(B) \ = \ 0 \qquad \Longleftrightarrow \qquad \mu(T^{-1}(B)) = 0.
  \end{equation}
  The map is measure preserving when the measure $T_*\mu$ defined by $T_* \mu(B) = \mu(T^{-1}(B))$ is equal to $\nu$. The map is
  null-preserving when $T_* \mu$ admits the Radon-Nikodym derivative $\frac{d T_* \mu}{d \nu} : Y \to \R_+$ such that
  $\mu(T^{-1}(B)) \ = \ \int_B  \frac{d T_* \mu}{d \nu} d \nu$. It is nonsingular when the Radon-Nikodym derivative is positive a.e.

  The map $T$ is a \emph{nonsingular isomorphism} when there exists a nonsingular map $\hat T : (Y,\nu) \to (X,\mu)$ such that
  $\hat T \circ T = id_X$ a.e. and $T \circ \hat T = id_Y$ a.e.

  \vspace{.25cm}

     If $T : (X,\mu) \to (Y,\nu)$ is a null-preserving map, then the map from $\B_Y$ to $\B_X$ given by $B \mapsto T^{-1}(B)$
    induces a map $T^* : \M_Y \to \M_X$, which is an example of a measure algebra $\s$-homomorphism. \vspace{.25cm}

    \begin{df}\label{df3.01x} For measure algebras $\M_X$ and $\M_Y$ a mapping
    $S : \M_Y \to \M_X$ is a \emph{measure algebra $\s$-homomorphism} when it satisfies:
    \begin{itemize}
   \item[(i)] $S(0) = 0$.

     \item[(ii)] $S(Y \setminus B) = X \setminus S(B)$ for all $ B \in \M_Y$.

     \item[(iii)] For any sequence $\{ B_n \} \in \M_Y, \ \ S(\bigcup_n B_n) \ = \ \bigcup_n S(B_n)$.
    \end{itemize}
    \end{df}\vspace{.25cm}

      It follows that  $S(Y) = X$ and  $S$ preserves countable intersections and Boolean sums.  If it is bijective, then the inverse mapping
    is clearly a $\s$-homomorphism from $\M_X$ to $\M_Y$ and so $S$ is then an isomorphism in the category of measure algebras and
    $\s$-homomorphisms.

 The translation to measure space results uses the same
   theorems that were described at the end of Section 3. In particular, it is a theorem of Sikorski that a $\s$-homomorphism comes from
   a null-preserving map, see \cite{R} Proposition 15.3 (Proposition 15.10 in the second edition)  \vspace{.25cm}

    \begin{lem}\label{lem3.01x} If a sequence $\{ B_n \}$ in $\M_Y$  converges to $B$, then there exists a subsequence $\{ B_{n_k} \}$
    such that $B \ = \ \bigcap_N \ \bigcup_{k \ge N}  B_{n_k}$, i.e. $B$ is the limsup of the subsequence. \end{lem}

    \begin{proof} Inductively, we choose an increasing sequence $\{ n_k \}$ such that $m \ge n_k$ implies $d(B_m, B) < 2^{-k}$.
    Let $\hat B_N = \bigcup_{k \ge N}  B_{n_k}$. Since $\hat B_N  + B \subset \bigcup_{k \ge N}  (B_{n_k} + B)$ it follows that
    $d(\hat B_N, B) < 2^{-N+1}$. Hence, $B$ is the limit of the decreasing sequence $\{ \hat B_N \}$.

    \end{proof} \vspace{.25cm}

    \begin{prop}\label{prop3.02x} A measure algebra $\s$-homomorphism $S : \M_Y \to \M_X$ is uniformly continuous and has a closed image.
     \end{prop}

    \begin{proof} If $S$ were not uniformly continuous, there would exists $\ep > 0$ and a sequence $\{ B_n \}$ in $\M_Y$ with
    $\nu(B_n) < 2^{-n}$ and $\mu(S(B_n)) \ge \ep$. Let $\hat B_k = \bigcup_{i \ge k}  B_{i}$. Since $\nu(\hat B_k) < 2^{-k+1}$ it follows
    that $\bigcap_k \hat B_k = 0$. However, $B_k \subset \hat B_k$ implies that $\{ S(\hat B_k) \}$ is a decreasing sequence in $\M_X$
    with $\mu(S(\hat B_k)) \ge \ep$ for all $k$. Hence, $ \mu(\bigcap_k \ S(\hat B_k)) \ge \ep$ and so $\bigcap_k \ S(\hat B_k) \not= 0 =
     S( \bigcap_k \ \hat B_k)$. This contradicts the assumption that $S$ is a $\s$-homomorphism.

     Now suppose that $A \in \M_X$ and $\{ B_n \}$ is a sequence in $\M_Y$ such that $\{ S(B_n) \}$ converges to $A$. By Lemma \ref{lem3.01x}
     we may assume, by going to a subsequence if necessary, that $A = \limsup S(B_n)$. Let $B = \limsup B_n$. Because $S$ commutes with countable
     operations it follows that $B \in \M_Y$ with $S(B) = \limsup S(B_n) = A$. Notice that we do not claim and do not need that $\{ B_n \}$
     converges to $B$. Nonetheless, we see that $S$ has a closed image.

    \end{proof} \vspace{.25cm}

    For a measure algebra $\s$-homomorphism $S : \M_Y \to \M_X$ we define the \emph{kernel of $S$} by
    \begin{equation}\label{eq3.02ax}
    Ker(S) \ = \ S^{-1}(0) \ = \ \{ A \in \M_X : S(A) = 0 \}
    \end{equation}\vspace{.25cm}
    We call the kernel trivial when it equals $\{ 0 \}$.

\vspace{.25cm}

       \begin{prop}\label{prop3.02xa} For a measure algebra $\s$-homomorphism $S : \M_Y \to \M_X$, the following are equivalent.
       \begin{itemize}
      \item[(a)] The map $S$ is injective.
       \item[(b)] The kernel $Ker(S)$ is trivial.
       \item[(c)] For every $\ep > 0$ there
    exists $\d > 0$ such that for all $B\in \M_Y$
    \begin{equation}\label{eq3.02x1}
   \nu(B) > \ep \quad \Longrightarrow \quad \mu(S(B)) \ge \d.
    \end{equation}
       \item[(d)] For every $\ep > 0$ there
    exists $\d > 0$ such that for all $B_1, B_2 \in \M_Y$
    \begin{equation}\label{eq3.02x2}
    d(B_1,B_2) > \ep \quad \Longrightarrow \quad d(S(B_1),S(B_2)) \ge \d.
    \end{equation}
    \end{itemize}
    The inverse map $S^{-1} : S(\M_Y) \to \M_Y$ is uniformly continuous.
    \end{prop}

    \begin{proof} (a) $\Leftrightarrow$ (b): $S(B_1) = S(B_2)$ if and only if $S(B_1 + B_2) = S(B_1) + S(B_2) = 0$.

    (c) $\Leftrightarrow$ (d) : $d(B_1,B_2) = \nu(B_1 + B_2)$ and $d(S(B_1),S(B_2)) = \mu(S(B_1 + B_2)) = \mu(S(B_1) + S(B_2))$.

    (c) $ \Rightarrow$ (b) : The contrapositive of (\ref{eq3.02x1}) clearly implies (b).

    (b) $\Rightarrow$ (c) : We mimic the proof of the first part of Proposition \ref{prop3.02x}. If (\ref{eq3.02x1}) fails,
    then there exists $\ep > 0$ and a sequence $\{ B_n \}$ in $\M_Y$ with
    $\nu(B_n) \ge \ep$ and $\mu(S(B_n)) < 2^{-n}$. Let $\hat B_k = \bigcup_{i \ge k}  B_{i}$. Since $\mu(S(\hat B_k)) < 2^{-k+1}$ it follows
    that $S(\bigcap_k \hat B_k) = 0$. However, $B_k \subset \hat B_k$ implies that $\{\hat B_k \}$ is a decreasing sequence in $\M_Y$
    with $\nu(\hat B_k) \ge \ep$ for all $k$ and so $\bigcap_k \hat B_k$ is a non-trivial element of the kernel.

    Uniform continuity of $S^{-1} : S(\M_Y) \to \M_Y$ is the contrapositive of (\ref{eq3.02x2}).

    \end{proof} \vspace{.25cm}

    Clearly, a null-preserving map $T$ is nonsingular exactly when $T^*$ is injective.\vspace{.25cm}

    \begin{prop}\label{prop3.02xb}  Assume $S : \M_Y \to \M_X$ is a measure algebra $\s$-homomorphism with $\M_Y$ separable.
    If the kernel of $S$ is non-trivial, then there exists $K \in \M_Y$ such that $Ker(S) = \M_K$. We call $K$ the \emph{zero-set} for $S$
    and denote it complement $Y \setminus K$ by $K'$.The restriction $S : \M_{K'} \to \M_X$ is injective with $S(\M_{K'}) = S(\M_Y)$.
    In particular, if $S$ is surjective, then the restriction $S : \M_{K'} \to \M_X$ is a $\s$-isomorphism. \end{prop}

    \begin{proof} Because $\M_Y$ is a separable metric space, the subset $Ker(S)$ contains a dense sequence $\{ K_i \}$. From continuity
    of $S, \ \ K = \bigcup_i K_i \in Ker(S)$. So if $B \subset K$, then $B \in Ker(S)$. Conversely, if $B \in Ker(S)$ and $\ep > 0$, then there exists
  $K_i$ in the sequence with $\mu(B + K_i) < \ep.$ Since $B \setminus K \subset B + K_i$ and $\ep$ is arbitrary, it follows that $\mu(B \setminus K) = 0$
  and so $B \subset K$.

    For any $B \in \M_Y$, $B = (B \cap K') \cup (B \cap K)$ and
    $B \cap K \in Ker(S)$. Hence, $S(B \cap K) = 0$ and so $S(B) = S(B \cap K')$. Thus,  $S(\M_{K'}) = S(\M_Y)$. If $B \subset K'$ and
    $B \not= 0$, then $B \not\in Ker(S)$ and so $S(B) \not= 0$. Hence, the restriction $S : \M_{K'} \to \M_X$  is injective.

    \end{proof} \vspace{.25cm}

  For measure algebras $\M_Y$ and $\M_X$ let $\Sigma(\M_Y,\M_X)$ be the set of measure algebra $\s$-homomorphisms from $\M_Y$ to $\M_X$ and let
  $\Sigma_{sur}(\M_Y,\M_X)$ and $\Sigma_{iso}(\M_Y,\M_X)$
  be the subset of surjections and isomorphisms, respectively.

  The coarse topology that we used for measure algebra maps in the previous section was that
  of pointwise convergence which identifies $\Sigma$ with a subset of
  $\M_X^{\M_Y}$ equipped with the product topology.
  Thus, a net $\{S_i \}$ converges
  to $S$ when $\{ S_i(B) \}$ converges to $S(B)$ in $\M_X$ for all $B \in \M_Y$.
  Now, however, the topology does not appear to be metrizable as
  pointwise convergence on a dense countable subalgebra may not be enough to assure convergence. Furthermore,
 if $S$ is a limit point for $\Sigma$ in the product, then it is a map
  from $\M_Y$ to $\M_X$ which preserves the finite Boolean operations, but it lies in $\Sigma$  if and only if it preserves countably infinite unions,
  or, equivalently via   Lemma  \ref{lem3.01x},
  if and only if it is continuous. Thus, the pointwise limit might not be a $\s$-homomorphism.

  To avoid the problems associated with the coarse topology, we turn back to the set of nonsingular
  maps with $(X,\mu)$ and $(Y,\nu)$ finite Polish measure spaces
  and consider their representation by $L^1$ isometries.

  Given a nonsingular map $T : (X,\mu) \to (Y,\nu)$ it is observed in \cite{KK} that the map $I_T : L^1(Y,\nu) \to L^1(X,\mu)$ defined by
   \begin{equation}\label{eq3.03xxa}
    I_T(f) \ = \ (f \circ T) \cdot (\frac{d\nu}{d T_*\mu} \circ T)
    \end{equation}
    is an isometry, because
     \begin{equation}\label{eq3.03xxb}
     \int \ (f \circ T) \cdot (\frac{d\nu}{d T_*\mu} \circ T) \ d \mu \ = \ \int \ f  \cdot \frac{d\nu}{d T_*\mu} \ d T_* \mu \ = \ \int \ f \ d \nu.
     \end{equation}

     It is a theorem of Lamperti, see \cite{L}, or \cite{R} Theorem 15.24 (Theorem 15.16 in the second edition), that every isometry $I : L^1(Y,\nu) \to L^1(X,\mu)$ is given by
     \begin{equation}\label{eq3.03xxc}
    I(f) \ = \ (f \circ T) \cdot h.
    \end{equation}
    Clearly, $h = I(1_Y)$ is uniquely determined in $L^1(X,\mu)$ by $I$. The measurable map $T$ is uniquely determined on the set $h \not= 0$
      (see the proof of Proposition \ref{prop3.07}).
    If $h \not= 0$ a.e., then $T$ is a nonsingular map.

    Note that if the isometry $J :  L^1(X,\mu) \to L^1(Z,\th)$ is given by $J(g) = (g \circ R) \cdot k$, then $J \circ I$
    is the isometry given by
    \begin{equation}\label{eq3.03xxcc}
    J(I(f)) = f \circ (T \circ R) \cdot (h \circ R) \cdot k.
    \end{equation}
     In particular, if $I$ is surjective and
    so is an isomorphism, then with $J = I^{-1}$ we see that $T \circ R = id_X$ and $(h \circ R) \cdot k = 1$. Since $I \circ J$ is the
    identity as well, it follows that $T$ is a nonsingular isomorphism with inverse $R$.

    Notice that if $h$ is of the form $h_1 \circ T$, then
      \begin{equation}\label{eq3.03xxd}
      \int \ f  \cdot \frac{d\nu}{d T_*\mu} \  d T_* \mu \ = \  \int \ f \cdot h_1  \ d T_* \mu.
      \end{equation}
      Since the measure $T_* \mu$ is equivalent to $\nu$, it follows that $h_1 = \frac{d\nu}{d T_*\mu}$ and so
      $h = \frac{d\nu}{d T_*\mu} \circ T$. In particular, if $T$ is a nonsingular isomorphism, then $h = h_1 \circ T$ with
      $h_1 = h \circ T^{-1}$. That is, for a nonsingular isomorphism $T$, the isometry $I_T$ of (\ref{eq3.03xxa} is the unique isometry with
      associated nonsingular map $T$ and it is an isometric isomorphism with $I_T^{-1}(g) = (g \circ T^{-1}) \cdot (1/h \circ T^{-1}).$

Suppose instead that $I$ is an isometry which is not surjective.  In that case, the image is a proper closed subspace of $L^1(X, \mu)$ and so
 there exists a non-zero $q \in L^{\infty}(X, \mu)$ such that $\int \ I(f) \cdot q \ d\mu = 0$ for all $f$. Normalize so that $|| q ||_{\infty} = 1$.
 For every $\ep \not= 0, \ I_{\ep}(f) = (f \circ T) \cdot h \cdot (1 + \ep q)$ is an isometry distinct from $I$.
This includes examples with $h$ negative on sets of positive measure.  On the other hand, assume that $h > 0$ a.e. If $|\ep| < 1$,
 then $h \cdot (1 + \ep q) > 0 $ a.e. while if $|\ep| = 1$, then $h \cdot (1 + \ep q) \ge 0 $  a.e. but may be $0$ on a
 set of positive measure.\vspace{.25cm}

 For example, let $T$ be the tent map on $[0,1]$ so that $I_T$ given by
 $f \mapsto f \circ T$ is an isometry, i.e. $h = 1$, because $T$ is measure preserving.
 Let  $q(t) = 2t - 1$ so that $\int_{T^{-1}E} q \ d\lm = 0$ for all $E$.
 Alternatively, for $A$ any subset of $[0,\frac{1}{2})$ with positive measure, let
  \begin{equation}\label{eq3.03xxdd}
  q(t) = \begin{cases} -1 \ \text{for}\ t \in A, \\ 1 \ \text{for}\ 1-t \in A,\\ 0 \ \text{otherwise}.\end{cases}.
     \end{equation}
In either case, for every $\ep \not= 0$,
 $I_{\ep} = (f \circ T)\cdot (1 + \ep q)$ is an isometry distinct from $I_T$.\vspace{.25cm}

 We define  $\J(L^1(Y,\nu),L^1(X,\mu))$  to be the set of isometries with $h \ge 0$ a.e. These are the \emph{positive isometries} such that
 $f \ge 0$ in $L^1(Y, \nu)$ implies $I(f) \ge 0$ in $L^1(X,\mu)$.

 For any $A \in \B_X$, the condition $\int \ 1_A h \ d \mu < 0$ is an open condition on $h \in L^1(\mu)$. Consequently, $\J(L^1(Y,\nu),L^1(X,\mu))$
 is a closed subset of the completely
 metrizable space of isometries with the strong topology such that $\{ I_i \}$ converges to $I$ when $\{ I_i(f) \}$ converges to
 $I(f)$ in $L^1(X,\mu)$ for all $f \in  L^1(Y,\nu)$. If  $\{ E_n \}$ is a dense sequence in $\M_Y$ then, since simple functions are dense in $L^1$,  we may use the metric
  \begin{equation}\label{eq3.03xxe}
  d(I_1,I_2) \ = \ \max_n \  \{ \frac{1}{n}|| I_1(1_{E_n}) - I_2(1_{E_n}) ||_1 \}
  \end{equation}

 Now if $I(f) = (f \circ T) \cdot h$ with $h \ge 0$ but $0$ on a set of positive measure, then, if necessary, we may redefine $T$ on the set
 $Z =  \{ h = 0 \}$ so that $T$ is nonsingular.  For example, we may choose $T$ to be an isomorphism of the normalization of $(Z,\mu|Z)$ to
 the normalization of  $(Y,\nu)$. With $T$ nonsingular,  the isometries $I_{\ep}(f) = (f \circ T) \cdot h_{\ep}$ with
 $h_{\ep} \ = \ (1 - \ep)h + \ep (\frac{d\nu}{d T_*\mu} \circ T)$ and $0 < \ep < 1$ all have $h_{\ep} > 0$. For $f = 1_{E_n}$ we have that
 $||I_{\ep}(f) - I(f)||_1\le \ep ||h - (\frac{d\nu}{d T_*\mu} \circ T)||_1.$ Hence, $I$ is the limit of the isometries $I_{\ep}$.
 This shows that the set of isometries with $h > 0$ a.e. is dense in $\J(L^1(Y,\nu),L^1(X,\mu))$.

  We define $\J_{iso}(L^1(Y,\nu),L^1(X,\mu))$ to be the set of isometric isomorphisms in $\J(L^1(Y,\nu),L^1(X,\mu))$.\vspace{.25cm}

  \begin{prop}\label{prop3.03xa}($\mathbf{G_{\d}}$) With $(X, \mu)$ and $(Y, \nu)$ Polish measure spaces, the
    set $\J_{iso}(L^1(Y,\nu),L^1(X,\mu))$ of isometric isomorphisms is a $G_{\d}$
  subset of $\J(L^1(Y,\nu),L^1(X,\mu))$.\end{prop}

 \begin{proof} Each isometry $I$ has a closed image and is an isomorphism if and only if it is surjective.
If $I$ fails to be surjective, then some $1_{E_n}$ is not in the image where $\{ E_n \}$ is a
 dense sequence in $\M_X$. As in the proof of \ref{prop3.04} the set
 \begin{equation}\label{eq3.05xxa}
\S_{nk} = \{ I : ||1_{E_n} - I(f) ||_1 \ge 1/k \ \text{for all} \ f \in L^1(Y,\nu) \}.
\end{equation}
is closed in $\J(L^1(Y,\nu),L^1(X,\mu))$. The set of isomorphisms \\ $\J_{iso}(L^1(Y,\nu),L^1(X,\mu))$ is the complement
of the union and so is a $G_{\d}$ set.

 \end{proof}\vspace{.25cm}

  \begin{prop}\label{prop3.03xb}(Density) With $(X, \mu)$ and $(Y, \nu)$ Polish measure spaces, if $I(f) = (f \circ T) \cdot h$ is an isometry
  in $\J(L^1(Y,\nu),L^1(X,\mu))$ with $h > 0$ a.e., then
  there exists a sequence $T_n : (X,\mu) \to (Y,\nu)$ of nonsingular isomorphisms such that $\{ T_n^* \}$ converges pointwise
   to $T^*$ and $I_n(f) = (f \circ T_n) \cdot h$ defines a sequence of isometric isomorphisms which converges strongly to $I$.

    The set $\J_{iso}(L^1(Y,\nu),L^1(X,\mu))$ of isometric isomorphisms is a dense
  subset of $\J(L^1(Y,\nu),L^1(X,\mu))$.\end{prop}

  \begin{proof}
We define $\mu_T(A)  = \int_A h \ d\mu$, obtaining a measure on $X$ with Radon-Nikodym derivative $\frac{d \mu_T}{d \mu} = h$. In addition,
 $T_* \mu_T = \nu$ because for all $B \in \B_Y$,
  \begin{equation}\label{eq3.05xxb}\begin{split}
  \int \ 1_B  \ d T_* \mu_T \ = \ \int \ (1_B \circ T) \ d \mu_T \ = \hspace{1cm}\\
   \int \ (1_B \circ T) \cdot h \ d \mu \ = \ \int I(1_B) \ d\mu  \ = \ \int 1_B \ d\nu.
  \end{split}\end{equation}

 If $T$ is an isomorphism, then the unique
 measure $\mu_T $ with $T_* \mu_T = \nu$ is $\nu \circ T = (T^{-1})_* \nu$. However, if $T$ is not an isomorphism, then
 different $h$'s can be used to get isometries and these lead to different measures $\mu_T$.

 We thus factor $T : (X,\mu) \to (Y,\nu)$ as the composition of
 the non-singular isomorphism $T_1 = id : (X,\mu) \to (X,\mu_T)$ with the measure space map (i.e. measure preserving map)
 $T_2 : (X,\mu_T) \to (Y,\nu).$ (which is just $T$ with a different measure). Thus, $T = T_2 \circ T_1$.

 This factors the isometry $I : L^1(Y,\nu) \to L^1(X,\mu)$ by the isometries $I_2 : L^1(Y,\nu) \to L^1(X,\mu_T)$ given by $ I_2(f) = f \circ T_2$
 and $I_1 : L^1(X,\mu_T) \to L^1(X,\mu)$ given by $I_1(g) = g \cdot h$. We have
  \begin{equation}\label{eq3.05xxc}\begin{split}
  \int \ I_2(f) \ d \mu_T = \int \ f \circ T \ d \mu_T = \int \ f \ d T_* \mu_T = \int \ f \ d \nu,\\
  \text{and} \qquad \int \ I_1(g) \ d\mu = \int \ g \cdot h \ d\mu = \int \ g \ d \mu_T, \hspace{1cm}
   \end{split}\end{equation}
with $I = I_1 \circ I_2$.

 By our measure space results, there exists a sequence of $T_n : (X,\mu_T) \to (Y,\nu)$ of measure space isomorphisms such that
 $T_n^*(E)$ converges to $T_2^*(E)$ for every Borel set $E$.  That is $\int 1_E \circ T_n \ d\mu_T = \mu_T(T_n^{-1}(E))$ converges to
 $\int 1_E \circ T_2 \ d\mu_T = \mu_T(T^{-1}(E)) $
 for every $E \in \B_Y$. Because $T_1$ is a nonsingular isomorphism, it follows from Proposition \ref{prop3.02x} that
 $\int 1_E \circ T_n \ d\mu = \mu(T_n^{-1}(E))$ converges to $\int 1_E \circ T_2 \ d\mu = \int 1_E \circ T \ d\mu =  \mu(T^{-1}(E))$. Thus, $\{ T_n \circ T_1 \}$ is a sequence of
 nonsingular isomorphisms converging pointwise to $T$.

  With $I_n : L^1(Y,\nu) \to L^1(X,\mu_T)$ given by $ I_n(f) = f \circ T_n$ we have that $I_n$ converges strongly to
 $I_2$, i.e. we have convergence with respect to the metric given by equation (\ref{eq3.03xxe}). Since $I_1$ is an isometry,
 it follows that $I_1 \circ I_n$ converges strongly to $I_1 \circ I_2 = I$. The composition
 $(I_1 \circ I_n)(f) = f \circ T_n \cdot h$. Because each $T_n$ is a measure space isomorphism, each $I_n$ is an isometric isomorphism.
 Also $I_1$ is an isometric isomorphism and so the compositions $I_1 \circ I_n$ are isomorphisms.

  As noted above the set of isometries with $h > 0$ is dense in \\ $\J(L^1(Y,\nu),L^1(X,\mu))$ and so the last statement follows.

 \end{proof} \vspace{.25cm}

 Thus, we have:

   \begin{theo}\label{theo3.03xc} Let $(X, \mu)$ and $(Y, \nu)$ be Polish measure spaces.
   The set $\J_{iso}(L^1(Y,\nu),L^1(X,\mu))$ of isometric $L^1$ isomorphisms is a dense, $G_{\d}$ subset of the
 completely metrizable space $\J(L^1(Y,\nu),L^1(X,\mu))$ of positive isometries,  i.e. those isometries with
    $I(1_Y) \ge 0$. That is, the generic positive isometry is invertible. \end{theo}\vspace{.25cm}

     \begin{cor}\label{cor3.03xbb} With $(X, \mu)$ and $(Y, \nu)$ Polish measure spaces,if $T : (X,\mu) \to (Y,\nu)$ is a null-preserving map,
      there exists a sequence $T_n : (X,\mu) \to (Y,\nu)$ of null-preserving maps such that $\{ T_n^* \}$ is a sequence of surjective
      $\s$-homomorphisms converging pointwise    to $T^*$.\end{cor}

  \begin{proof} The restriction $T' : (X,\mu) \to (K',\nu|K')$
    is nonsingular where $K'$ is the complement of the zero-set for $T$. For the inclusion map
    $inc : (K',\nu|K') \to (Y,\nu)$ the measure algebra $\s$-homomorphism
    $inc^*$ is surjective and $T = inc \circ T'$.
    Proposition \ref{prop3.03xb} gives a sequence of nonsingular isomorphisms $T'_n: (X,\mu) \to (K',\nu|K')$ such that
    $\{ (T_n')^{*} \}$ converges  pointwise to $(T')^{*}$.  So $T_n = inc \circ T'_n: (X,\mu) \to (Y,\nu)$ is a sequence of null-preserving maps
    with $\{ T_n^* \}$ a sequence of surjections converging pointwise to $T^*$.

 \end{proof} \vspace{.5cm}

    For the $\s$-homomorphisms associated with null-preserving maps we use the uniform topology on $\Sigma(\M_Y,\M_X)$ with the
  the sup metric.
  \begin{equation}\label{eq3.03x}
  d_u(S_1,S_2) \ = \ \sup \{ d(S_1(A),S_2(A)) : A \in \M_Y \}
  \end{equation}
  Since the uniform limit of continuous functions is continuous it follows that $\Sigma(\M_Y,\M_X)$ is a closed subset of the complete,
  separable metric space of continuous functions from $\M_Y$ to $\M_X$.

  This notion of convergence is very restrictive as illustrated by the following examples.
      \vspace{.25cm}

      \begin{exes}\label{exes3.02xxb} Examples on the unit interval. \end{exes}

      (a) For $\ell\in (0,1)$ define $\tilde T_{\ell} : I \to I$ by $\tilde T_{\ell}(t) = \ell \cdot t$. One would like $\tilde T_{\ell}^*$ to
      converge to the identity on $\M_I$ as $\ell \to 1$. However, for all such $\ell$ we have $d_u(\tilde T_{\ell}^*,id_{M_I}) = 1$.
      To see this, let $A_{\ell} = \bigcup_{k=0}^{\infty} (\ell^{2k+2},\ell^{2k+1}]$. Note that $\tilde T_{\ell}^{-1}(A_{\ell})$ is
      disjoint from $A_{\ell}$ and their union is $(0,1]$. Hence, $d(\tilde T_{\ell}^*(A_{\ell}),A_{\ell}) = 1$.

      (b) Similar problems arise even for very nice maps. For $n \in \N$ the map $\hat T_n : I \to I$ by $\hat T_n(t) = t + \frac{1}{2n} \  (mod \ 1)$
      is a measure preserving isomorphism.
      Let $A_n = \bigcup_{k=1}^{n} \ (\frac{2k-1}{2n},\frac{2k}{2n}]$ so that $\hat T_n^{-1}(A_n)$ is disjoint
      from $A_{n}$ and their union is $(0,1]$. Hence, $d(\hat T_{n}^*(A_{n}),A_{n}) = 1$ and so $d_u(\hat T_{n}^*,id_{M_I}) = 1$.

       $\Box$ \vspace{.25cm}

  \begin{prop}\label{prop3.03xd}($\mathbf{G_{\d}}$) If $\M_Y$ and $\M_X$ are separable, nonatomic measure algebras, then the set of
  surjections $\Sigma_{sur}(\M_Y.M_X)$  is
  closed with respect to the uniform topology. The set of
isomorphisms\\ $\Sigma_{iso}(\M_Y.M_X)$ is a $G_{\d}$ subset of $\Sigma_{sur}(\M_Y.M_X)$.\end{prop}

\begin{proof}
A map $S \in \Sigma$ fails to be surjective when some $A \in \M_X$ is a positive distance from the image. For any such $A$ define
 \begin{equation}\label{eq3.04xa}
   G_A\ = \ \{ S : \text{for some } \ \d > 0, \ d(A, S(B)) > \d \ \ \text{for all} \ \ B \in \M_Y \}.
   \end{equation}
   This is open, because if $d(A, S(B)) > \d$ and $d_u(S,S_1) < \d/2$, then $d(A, S_1(B)) > \d/2$. The union of the $G_A$'s is the complement of $\Sigma_{sur}$.

   For injectivity, let $\{ \nu \ge 1/k \}$ denote the subset $\{ B \in \M_Y : \nu(B) \ge 1/k \}$ and similarly define $\{ \mu \le \d \}$ for $\d > 0$.
   Define
    \begin{equation}\label{eq3.04xb}
   G_k \ = \ \{ S : \text{for some } \ \d > 0, \ S(\{ \nu \ge 1/k \}) \cap \{ \mu \le \d \}  = 0 \}.
   \end{equation}

   This is open because if $S(\{ \nu \ge 1/k \}) \cap \{ \mu \le \d \}  = 0$ and $d(S,S_1) < \d/2$, then
   $S(\{ \nu \ge 1/k \}) \cap \{ \mu \le \d/2 \}  = 0.$  That is, if $B \in \{ \nu \ge 1/k \}$ and $\mu(S_1(B)) < \d/2$, then
   $\mu(S(B) + S_1(B)) < \d/2$ implies $\mu(S(B)) < \d.$  From Proposition \ref{prop3.02xa} it follows that
   $\bigcap_k G_k$ is the set of injective $\s$-homomorphisms.

Thus, the bijective maps in $\Sigma$ comprise a $G_{\d}$ set for the uniform topology.
By Proposition \ref{prop3.02xa} $\Sigma_{iso}$ is the set of bijective maps in $\Sigma$.

    \end{proof} \vspace{.25cm}

    Thus, the tent map $T$ given by (\ref{eq1.01}) cannot be uniformly approximated by a non-singular isomorphism because $T^*$ is not surjective.

     \begin{prop}\label{prop3.06x}    Let $(X, \mu)$ and $(Y, \nu)$ be Polish measure spaces.
     Given a null-preserving map $T : (X,\mu)\to (Y,\nu)$ with $T^*$ surjective
   there exists a sequence $\{ T_n \}$ of nonsingular isomorphisms from $(X,\mu)$ to $(Y,\nu)$ such that
   the sequence of $\s$-homomorphisms $\{ T_n^* \}$ converges to $T^*$ uniformly.  \end{prop}

    \begin{proof} If the zero-set is $0$, then $T$ is already a nonsingular isomorphism. In any case, Proposition \ref{prop3.02xb}
    implies that $T$
    restricts to a nonsingular isomorphism  $T' : (X,\mu) \to (K',\nu|K')$. With $inc : (K', \nu|K') \to (Y,\nu)$
    we have $inc^*(B) = B \cap K'$. Thus, $T$ factors as $inc \circ T'$.

    Choose disjoint subsets $A_n, B_n \subset K'$ with $e_n = \nu(A_n) = \nu(B_n) > 0$ and $1/n > e_n$. Choose $\hat T_n: (B_n,\nu|B_n) \to (K,\nu|K)$
    to be an isomorphism of the normalized measures spaces, $\hat T_n :  (A_n,\nu|A_n) \to (A_n \cup B_n,\nu|(A_n \cup B_n))$
     to be an isomorphism of the normalized measures spaces and $\hat T_n $ to be the identity on $ K' \setminus (A_n \cup B_n)$.
    For $B \in \B_Y$, $inc^*(B) + \hat T_n^*(B) \subset A_n \cup B_n$ and so has measure less than $2/n$. The Radon-Nikodym derivative
    $\frac{d (\hat T_n)_*(\nu|K')}{d \nu}$ is given by $e_n/\nu(K)$ on $K, \ 1/2$ on $ A_n \cup B_n$ and $1 $ on $ K' \setminus (A_n \cup B_n)$.

    By Proposition \ref{prop3.02x} the map $T'^*$ is uniformly continuous. So for every $\ep > 0$ there exists
    $\d > 0$ such that for $B_1, B_2 \subset K', \    d(B_1, B_2) < \d$ implies $d(T'^*(B_1), T'^*(B_2)) < \ep$.
    If $T_n = \hat T_n \circ T'$ then $2/n < \d$ will imply that for $B \in \B_Y$
    \begin{equation}\label{eq3.04xc}
    d(T_n^*(B), T^*(B)) \ = \ d(T'^*(\hat T_n^*(B)), T'^*(inc^*(B))) \ < \ \ep.
    \end{equation}

    Thus, the sequence $\{ T_n \}$ of nonsingular isomorphisms has $T_n^*$ converging uniformly to $T^*$.

    \end{proof}  \vspace{.25cm}

    We call $T : (X, \mu) \to (Y, \nu)$ an \emph{injection} if there exists $B \in \B_Y$ and a nonsingular isomorphism $T' : (X,\mu) \to (B, \nu|B)$
    such that $T = inc \circ T'$ with $inc : (B, \nu|B) \to (Y,\nu)$ the inclusion map. \vspace{.25cm}

   \begin{prop}\label{prop3.06xbbc}  A null-preserving map $T : (X,\mu)\to (Y,\nu)$  with $(Y,\nu)$ separable is an injection if and only if
   $T^* : \M_Y \to \M_X$ is a surjection. \end{prop}

   \begin{proof} We saw in the proof of Proposition \ref{prop3.06x} that if $T^*$ is surjective, then $T = inc \circ T'$ with $inc : (K', \nu|K') \to (Y,\nu)$
    the inclusion of the complement of the zero-set and with $T'$ an isomorphism. So $T$ is an injection.

    Conversely, if $T = inc \circ T'$ with $inc : (B, \nu|B) \to (Y,\nu)$ the inclusion and $T' : (X,\mu) \to (B, \nu|B)$ is a nonsingular isomorphism,
    then $(T')^* : \M_B \to \M_X$ is a $\s$-isomorphism and so is surjective and $inc^* : \M_Y \to \M_B$ is a retraction and so is surjective.
    Hence the composition $(T')^* \circ inc^* = (inc \circ T')^* = T^*$ is surjective.

    \end{proof}  \vspace{.25cm}

  Proposition \ref{prop3.07} holds for null-preserving maps and so if we denote by $\Sigma((X,\mu),(Y,\nu))$ the set of null-preserving maps from
  $(X,\mu)$ to $(Y,\nu))$, then because $T \mapsto T^*$ is a bijection from $\Sigma((X,\mu),(Y,\nu))$ to $\Sigma(\M_Y,\M_X)$, we may use this to
  transfer the coarse topology to $\Sigma((X,\mu),(Y,\nu))$. If we use it to transfer the uniform topology, then the set of injections, $\Sigma_{inj}((X,\mu),(Y,\nu))$,   and the set of isomorphisms, $\Sigma_{iso}((X,\mu),(Y,\nu))$, are mapped
  onto the closed set $\Sigma_{sur}(\M_Y,\M_X)$ and the $G_{\d}$ set
  $\Sigma_{iso}(\M_Y,\M_X)$, respectively. \vspace{.25cm}

    So, as usual, we obtain

     \begin{theo}\label{theo3.07x}   Let $(X, \mu)$ and $(Y, \nu)$ be Polish measure spaces.
        The set $\Sigma_{iso}((X,\mu),(Y,\nu))$ of nonsingular isomorphisms
        is a dense, $G_{\d}$ subset of the
        completely metrizable space $\Sigma_{inj}((X,\mu),(Y,\nu))$ of
       null-preserving injections, equipped with the topology of uniform convergence. That is, the generic
      null-preserving injection  is invertible  \end{theo}\vspace{.25cm}

       Finally, combining Corollary \ref{cor3.03xbb} with Proposition \ref{prop3.06x} we have \vspace{.25cm}

          \begin{theo}\label{theo3.08x}   Let $(X, \mu)$ and $(Y, \nu)$ be Polish measure spaces.
        The set $\Sigma_{iso}((X,\mu),(Y,\nu))$ of nonsingular isomorphisms
       is a dense subset of the space $\Sigma((X,\mu),(Y,\nu))$ of  null-preserving maps equipped with the coarse topology of pointwise convergence.
       In fact for any null-preserving map $T : (X, \mu) \to (Y, \nu)$ there is a sequence $T_n : (X, \mu) \to (Y, \nu)$
       of nonsingular isomorphisms converging pointwise to $T$.\end{theo}

          \begin{proof} With $K'$ the complement of the zero-set of $T$,
           the restriction $T : (X,\mu) \to (K',\nu|K')$ is a nonsingular map, and so there exists a measure $\mu_T$ on $X$ such that
          $T_1 = id_X : (X,\mu) \to (X,\mu_T)$ is a nonsingular isomorphism and $T_*\mu_T = \nu|K'$. Hence, $T_2 = T : (X,\mu_T) \to (K',\nu|K')$ is
          a measure preserving map. With $T_3 = inc : (K',\nu|K') \to (Y,\nu)$ the inclusion map, we have the factorization
          $T = T_3 \circ T_2 \circ T_1$.

          By our previous results, there exists a sequence $T_{2n}: (X,\mu_T) \to (K',\nu|K')$ of measure space isomorphisms which converge
         to $T_2$ and a sequence $T_{3n} : (K',\nu|K') \to (Y,\nu)$ of nonsingular isomorphisms which converge to $T_3$.

         Let $B \in \B_Y$. From convergence of the $T_3$ sequence we have \\ $d(T_{3n}^*(B),T_3^*(B)) \to 0$. Because each $T_{2n}$ is measure
         preserving we have $d(T_{2n}^*T_{3n}^*(B),T_{2n}^*T_3^*(B)) \to 0$. From convergence of the $T_2$ sequence
         we have $d(T_{2n}^*T_{3}^*(B),T_{2}^*T_3^*(B))) \to 0$ and so by the triangle inequality
         $d(T_{2n}^*T_{3n}^*(B),T_{2}^*T_3^*(B))) \to 0$. Here the metric on $\M_X$ uses the measure $\mu_T$. But because the measure
          $\mu_T$ is equivalent to $\mu$ we have $d(T_1^*T_{2n}^*T_{3n}^*(B),T_1^*T_{2}^*T_3^*(B))) \to 0$.

          Thus, $T_n = T_{3n} \circ T_{2n} \circ T_1$ is a sequence of nonsingular isomorphisms which converges to $T =  T_{3} \circ T_{2} \circ T_1$
          with respect to the coarse topology.

    \end{proof}  \vspace{.25cm}

       Of course, we don't have genericity results in this case because of various problems with the coarse topology.

    \vspace{.5cm}

 \section{\textbf{Good Measures on Cantor Spaces}}\vspace{.25cm}

 We will call a pair $(X,\mu)$ a \emph{Cantor measure space} when $\mu$ is a finite, full, nonatomic Borel measure on a Cantor space $X$.  Thus, every nonempty
 open subset has positive measure and every countable subset has measure zero. The clopen subsets form a countable algebra which is a basis for the topology and
 which generates the Borel $\s$-algebra. Whenever we need one, we will fix a metric $d$ on $X$. For a clopen partition we will call the maximum diameter the
 $d$-mesh and the maximum measure the $\mu$-mesh.  \vspace{.25cm}

 \begin{lem}\label{lem4.01} For every $\ep > 0$ there exists $\d > 0$ such that any Borel set with diameter less than $\d$ has measure less than $\ep$.
 In particular, if a clopen partition has $d$-mesh less than $\d$, then it has $\mu$-mesh less than $\ep$.
\end{lem}

\begin{proof} If $\{ A_i \}$ were a sequence of Borel sets with diameter tending to zero but with measures at least $\ep$, then for any limit point $x$,
every neighborhood would have measure at least $\ep$ and so $x$ would be an atom.

\end{proof}\vspace{.25cm}

A Cantor measure space map $T : (X,\mu) \to (Y,\nu)$ is a continuous map from $X$ to $Y$ which preserves measure.  To see that measure is preserved,
it suffices to check:
\begin{equation}\label{eq4.01}
\mu(T^{-1}(A)) = \nu(A) \ \text{for all clopen subsets} \ A \ \text{of} \ Y.
\end{equation}
If $T$ is a Cantor measure space map with $T$ a homeomorphism, then $T^{-1}$ satisfies (\ref{eq4.01}) and so $T$ is a Cantor measure space isomorphism. Notice
that if continuous maps $T_1, T_2 : X \to Y$ are equal a.e., then they are equal everywhere because $\{ x : T_1(x) \not= T_2(x) \}$ is an open set and
so has positive measure if it is nonempty. Thus, conversely, if  $T : (X,\mu) \to (Y,\nu)$ is a Cantor measure space isomorphism, then
$T$ is a homeomorphism from $X$ to $Y$.

Since $T^{-1}(Y) = X$, the existence of such a map requires $\mu(X) = \nu(Y)$. Furthermore, $T(X) = Y$ because $\nu$ is full and so $T(X)$ meets every
nonempty open subset. Thus, $T$ is necessarily surjective. \vspace{.25cm}

\begin{df}\label{df4.02} For a Cantor measure space $(X,\mu)$ the \emph{clopen values set} $S(\mu)$ is defined by
\begin{equation}\label{eq4.02}
S(\mu) = \{ \mu(A) : \ \text{for} \ A \ \text{clopen in} \ X \}.
\end{equation}\end{df}\vspace{.25cm}

\begin{prop}\label{prop4.03} (a) If $(X,\mu)$ is a Cantor measure space, then $S(\mu)$ is a countable, dense subset of $[0,\mu(X)]$ which contains the
endpoints.

(b) If $T : (X,\mu) \to (Y,\nu)$ is a Cantor measure space map, then $S(\nu) \subset S(\mu)$ with equality if $T$ is an isomorphism.
\end{prop}

\begin{proof} (a) Given $\ep > 0$ and $\d > 0$ from Lemma \ref{lem4.01} we choose $\{A_1, \dots, A_n \}$ is a clopen partition with $d$-mesh less than $\d$
then
$$\{ a_0 = 0, a_1 = \mu(A_1), a_2 = \mu(A_1 \cup A_2), \dots, a_n = \mu(A_1 \cup \dots \cup A_n) = \mu(X) \}$$
is an increasing sequence in $[0,\mu(X)]$ with $a_n - a_{n-1} < \ep$ for $n = 1, \dots, n$. As $\ep$ was arbitrary, it follows that $S(\mu)$ is dense.
It is clearly countable.

(b) This is obvious from (\ref{eq4.01}).

\end{proof} \vspace{.25cm}

Our results will be applied to the Cantor measure spaces which satisfy a homogeneity condition, the \emph{Subset Condition}. \vspace{.25cm}

\begin{df}\label{df4.04} For a Cantor measure space $(X,\mu)$, the measure $\mu$ is called \emph{good} and $(X,\mu)$ is called a
good Cantor measure space when for every pair of clopen subsets $A, B$ of $X$
\begin{equation}\label{eq4.03}\begin{split}
\mu(A) \le \mu(B) \quad \Longrightarrow \hspace{3cm} \\ \text{there exists a clopen} \ A_1 \subset B \ \text{with} \ \mu(A_1) = \mu(A).
\end{split}\end{equation}\end{df}\vspace{.25cm}

It is clear that if $(X,\mu)$ is a good Cantor measure space and $A$ is a nonempty clopen subset of $X$, then $(A,\mu|A)$ is a good Cantor
measure space with
\begin{equation}\label{eq4.04a}
S(\mu|A) = S(\mu) \cap [0,\mu(A)].
\end{equation}

A subset $S$ of an interval $[0,m]$ is called \emph{group-like} when $m \in S$ and $S$ is the
intersection of the interval with an additive subgroup of $\R$.\vspace{.25cm}

\begin{lem}\label{lem4.05} If $S$ is a subset of the unit interval $[0,1]$ with $0, 1 \in S$, then $(S + \Z) \cap [0,1] = S$. Furthermore,
$S + \Z$ is a subgroup of $\R$, and so $S$ is group-like, if and only if for all $s, t \in S$
\begin{equation}\label{eq4.04}
%s + t \le 1 \qquad \Longright \qquad s+t \in S,
s \le t \qquad \Longrightarrow \qquad t - s  \in S.
\end{equation} \end{lem}

\begin{proof} If $0 < t < 1$ and $n \in \Z$, then $n \ge 1$ implies $n + t > 1$ and $n \le -1$ implies $n +t < 0$.  Hence, $n+t \in [0,1]$ implies
$n = 0$. Thus, $(S + \Z) \cap [0,1] = S$.

If $S$ is group-like then (\ref{eq4.04}) obviously holds. Assume, conversely, that (\ref{eq4.04}) is true. Since $1 \in S$, $1 - t \in S$ when
$t \in S$. If $s + t \le 1$, then $s \le 1 - t$ and so $s + t = 1 - [(1 - t) - s] \in S$.    If $s + t > 1$, then
$(1 - s) + (1 - t) = 2 - (s + t) \le 1$. Hence, $s + t - 1 = 1 - [(1 - s) + (1 - t)] \in S$. It follows that
$S + \Z$ is closed under addition. Furthermore, $-(n + t) = -(n+1) + (1 - t)$ and so $S + \Z$ is closed under negation.

\end{proof} \vspace{.25cm}

We now describe the crucial properties of good measures.\vspace{.25cm}

\begin{theo}\label{theo4.06} (a) If $(X,\mu)$ is a good Cantor measure space, then $S(\mu)$ is group-like.

(b) If $S$ is a dense, group-like subset of $[0,1]$ with $0, 1 \in S$, then there exists a normalized good Cantor measure space with
$S(\mu) = S$.

(c) Assume that $(X,\mu)$ and $(Y,\nu)$ are good Cantor measure spaces with $\mu(X) = \nu(Y)$.

There exists a Cantor measure space isomorphism $T : (X,\mu) \to (Y,\nu)$ if and only if $S(\nu) = S(\mu)$.

There exists a Cantor measure mapping $T : (X,\mu) \to (Y,\nu)$ if and only if $S(\nu) \subset S(\mu)$.
  \end{theo}

\begin{proof} (a) A subset $S$ is group-like in $[0.m]$ if and only if $(1/m)S$ is group-like in $[0,1]$.
So we may assume $(X,\mu)$ is normalized. Let $s = \mu(A)$ and $t = \mu(B)$ with clopens $A, B$.  If $s \le t$ then there exists $A_1 \subset B$ clopen
with $\mu(A_1) = \mu(A)$.  Then $\mu(B \setminus A_1) = t - s$ and so (\ref{eq4.04})  holds for $S(\mu)$ and $S(\mu)$ is group-like
by Lemma \ref{lem4.05}.

(b) This result is Theorem 2.6 of \cite{A05} taken together with Corollary 3.3 of \cite{A99}.

(c) The conditions on the clopen values sets are  necessary by Proposition \ref{prop4.03}. By normalizing, i.e.
dividing the measures by $\mu(X) = \nu(Y)$ we may assume they are
normalized. The results are then in Theorem 2.9 of \cite{A05}.

\end{proof} \vspace{.25cm}

Thus, when we restrict ourselves to good Cantor measure spaces, the clopen values set is a complete invariant.  Furthermore, if $(X,\mu)$ and
$(Y,\nu)$ are good with $S(\mu) = S(\nu)$ and if $A \subset X$ and $A_1 \subset Y$ are clopens with $\mu(A) = \nu(A_1 )$, then there exists
an isomorphism from $(A,\mu|A)$ to $(A_1 ,\nu|A_1 )$.  The maximum values in $S(\mu) = S(\nu)$ are $\mu(X) = \nu(Y)$ and so
 we can combine the isomorphism from $A$ to $A_1 $ with an isomorphism
from $X \setminus A$ to $Y \setminus A_1 $ to obtain an isomorphism from $(X,\mu)$ to $(Y,\nu)$ which maps $A$ to $A_1 $. In
particular, with $X = Y$ we see that the homogeneity condition defining
a good measure is not as mild as it first appears.

Theorem \ref{theo4.06} provides the uniqueness results that we need in this context. \vspace{.25cm}

In Section 2 we defined $\CC(X,Y)$ to be the space of continuous maps with the topology of uniform convergence and in Section 3
we defined $\S((X,\mu),(Y,\nu))$ to be the set of measure space maps. For Cantor measure spaces the set $\CC \S((X,\mu),(Y,\nu))$
is the intersection $\CC(X,Y) \cap \S((X,\mu),(Y,\nu))$. We use the topology of uniform convergence on \\ $\CC \S((X,\mu),(Y,\nu))$ so it becomes
a subspace of $\CC(X,Y)$.\vspace{.25cm}

\begin{prop}\label{prop4.07}($\mathbf{G_{\d}}$) Let $(X,\mu)$ and $(Y,\nu)$ be Cantor measure spaces. The set $\CC \S((X,\mu),(Y,\nu))$
of Cantor measure space maps
is a closed subset of $\CC_{sur}(X,Y)$, the set of continuous surjections. The set \\ $\CC \S_{iso}((X,\mu),(Y,\nu))$ of Cantor
measure space isomorphisms is a $G_{\d}$ subset.
\end{prop}

\begin{proof} We saw above that any Cantor measure space map is surjective. A continuous function $T : X \to Y$ preserves measure exactly when
for every continuous function $f : Y \to \R, \ \int (f \circ T) \  d\mu = \int f \ d\nu$. This is clearly a closed condition and so
$\CC \S((X,\mu),(Y,\nu))$, the
set of Cantor measure space maps is a closed subset of $\CC_{sur}(X,Y)$.

By   Proposition \ref{prop2.03} $\CC_{iso}(X,Y)$ is a $G_{\delta}$ subset and so the intersection $\CC \S((X,\mu),(Y,\nu)) \cap \CC_{iso}(X,Y)$,
is a $G_{\delta}$ subset.  The intersection clearly contains the set of Cantor measure space isomorphisms. On the other hand, if $T : (X,\mu) \to (Y,\nu)$
is a Cantor measure space map with $T : X \to Y$ a homeomorphism then, as remarked above, $T$ is an isomorphism.  That is, the intersection equals the
set of isomorphisms.

\end{proof} \vspace{.25cm}

Now suppose that $(X,\mu)$ is a normalized, good Cantor measure space such that $t \in S(\mu)$ implies $t/2 \in S(\mu)$.  For example, this is true if
$S(\mu) = \Q \cap [0,1]$.  In that case, $2 \times ([0,\frac{1}{2}] \cap S(\mu)) = S(\mu)$.  If $A_1$ is a clopen set with $\mu(A_1) = \frac{1}{2}$ and
$A_2 = X \setminus A_1$,
then there exist isomorphisms $T_1 :  (A_1,2\mu|A_1) \to (X,\mu) $ and $T_2 :  (A_2,2\mu|A_2) \to (X,\mu)$. Combining these we obtain
a measure space map $T : (X,\mu) \to (X,\mu)$ which is the analogue in this context of the tent map.  \vspace{.25cm}

\begin{prop}\label{prop4.08}(\textbf{Density}) Let $T : (X,\mu) \to (Y,\nu)$ be a Cantor measure space map with $S(X,\mu) = S(Y,\nu)$.  For a metric $d$
 on $Y$ and $\ep > 0$, there exists a Cantor measure space isomorphism $T_1 :(X,\mu) \to (Y,\nu)$ such that
 \begin{equation}\label{eq4.08}
 x \in X \qquad \Longrightarrow \qquad d(T(x),T_1(x)) \le \ep.
 \end{equation}. \end{prop}

 \begin{proof} The proof is almost the same as that of Proposition \ref{prop2.04}.
 Let $\P = \{U_1, \dots, U_n \}$ be a clopen partition of $Y$ with $d$-mesh at most $\ep$. As before
 $ \{ T^{-1}(U_1), \dots, T^{-1}(U_n) \}$ is a clopen partition of $X$.
For each $i$,  $(U_i,\nu|U_i)$ and  $(T^{-1}(U_i),\mu|T^{-1}(U_i)),$ are good Cantor measure spaces with the same clopen values set and so
 Theorem \ref{theo4.06}  implies we can
 choose an isomorphism $\hat T_i $ from $ (T^{-1}(U_i),\mu|T^{-1}(U_i))$ onto $(U_i,\nu|U_i)$. Concatenating we obtain the isomorphism
  $T_1 : (X,\mu) \to (Y,\nu)$. Again
 for all $x \in X$
 \begin{equation}\label{eq4.09}
 T(x) \in U_i \qquad \Longleftrightarrow \qquad T_1(x) \in U_i.
 \end{equation}
and so (\ref{eq4.08}) follows.

 \end{proof}\vspace{.25cm}

 Thus, we obtain

 \begin{theo}\label{theo4.09}If $(X,\mu)$ and $(Y,\nu)$ are good Cantor measure spaces with the same clopen values set, then the set
 $\CC \S_{iso}((X,\mu),(Y,\nu))$
 of Cantor measure isomorphisms from $(X,\mu)$ to $(Y,\nu)$ is a dense $G_{\d}$ subset of the completely metrizable space $\CC \S((X,\mu),(Y,\nu))$of
 Cantor measure space maps from $(X,\mu)$ to $(Y,\nu)$. Thus, the generic Cantor measure space map between them is invertible.\end{theo} \vspace{.25cm}

\textbf{ Notes:} The theory of normalized good Cantor measure spaces is laid out in \cite{A99} and \cite{A05}. In Theorem 2.18 of \cite{A05} an
example is given of two Bernoulli measures with the same group-like clopen values set one of which is good and the other is not and so they are not isomorphic.
In \cite{ADMY} those Bernoulli measures which are good are characterized. \vspace{.5cm}

  \bibliographystyle{amsplain}

\end{document}